\documentclass[12pt,a4paper]{article}

\usepackage{amsmath}
\usepackage{amsfonts}
\usepackage{amssymb}
\usepackage{amsthm}

\theoremstyle{definition}

\newtheorem{lemma}{Lemma}[section]
\newtheorem{theorem}[lemma]{Theorem}
\newtheorem{claim}[lemma]{Claim}

\newtheorem{remark}[lemma]{Remark}
\newtheorem{definition}[lemma]{Definition}
\newtheorem{proposition}[lemma]{Proposition}

\renewenvironment{proof}{\textbf{Proof.}}{\qed}

\begin{document}

\title{Simultaneous Universal Pad\'{e} Approximation} 
\author{K. Makridis and V. Nestoridis \thanks{Dedicated to Professor C. Gryllakis for his $60^{th}$ birthday.}}
\maketitle

\begin{abstract}

We prove simultaneous universal Pad\'{e} approximation for several universal Pad\'{e} approximants of several types. Our results are generic in the space of holomorphic functions, in the space of formal power series as well as in a subspace of $A^{\infty}$. These results are valid for one center of expansion or for several centers as well. 

\end{abstract}

AMS clasification numbers: 30K05.

\medskip


\section{Introduction} 

Let $\Omega \subseteq \mathbb{C}$ be a simply connected domain, $\zeta \in \Omega$ and $\mu$ be an infinite subset of $\mathbb{N}$. A holomorphic function $f \in H(\Omega)$ is called universal Taylor series if for every compact set $K \subseteq \mathbb{C} \setminus \Omega$ with connected complement and for every funcion $h \in A(K)$, there exists a sequence $(\lambda_n)_{n \in \mathbb{N}} \subseteq \mu$ satisfying the following:

\begin{itemize}

\item[(i)]
$\sup_{z \in K} |S_{\lambda_n}(f, \zeta)(z) - h(z)| \to 0$ as $n \to + \infty$.

\item[(ii)]
$\sup_{z \in J} |S_{\lambda_n}(f, \zeta)(z) - f(z)| \to 0$ as $n \to + \infty$ for every compact set $J \subseteq \Omega$.

\end{itemize}

Here $S_N(f, \zeta)(z) = \sum_{j = 0}^N \frac{f^{(j)}(\zeta)}{j!}(z - \zeta)^{j}$ is the Taylor expansion of the function $f$ centered at $\zeta \in \Omega$. This is a generic property of holomorphic functions (\cite{NESTORIDIS4}, \cite{NESTORIDIS2}, \cite{MELAS.NESTORIDIS}). 

Recently, the partial sums have been replaced by some rational functions, the Pad\'{e} approximants of $f$.

Let $f(z) = \sum_{n = 0}^{+ \infty} a_n (z - \zeta)^{n}$ be a formal power series with center $\zeta \in \Omega$ and $p, q \in \mathbb{N}$. Then the $(p, q)$ - Pad\'{e} approximant of $f$ with center $\zeta \in \Omega$ is a rational function $[f; p / q]_{\zeta}(z) = \frac{A(z)}{B(z)}$, where the polynomials $A$ and $B$ satisfy $degA \leq p$, $degB \leq q$, $B(\zeta) = 1$ and the Taylor expansion of the function $\frac{A(z)}{B(z)} = \sum_{n = 0}^{+ \infty} b_n (z - \zeta)^{n}$ satisfies $a_n = b_n$ for every $n \leq p + q$. Such a rational function may exists or not, but if it exists it is unique. Furhtermore, when the function $[f; p / q]_{\zeta}$ exists and the polynomials $A$ and $B$ are unique, we write $f \in D_{p, q}(\zeta)$. There are two types of universal Pad\'{e} approximants.

\bigskip

\textbf{Type I} (\cite{DARAS.NESTORIDIS}, \cite{DARAS.FOURNODAULOS.NESTORIDIS}) Let $(p_n)_{n \in \mathbb{N}}, (q_n)_{n \in \mathbb{N}} \in \mathbb{N}$ with $p_n \to + \infty$, $\Omega \subseteq \mathbb{C}$ be a simply connected domain and $\zeta \in \Omega$ be a fixed point. A holomorphic function $f \in H(\Omega)$ with Taylor expansion $f(z) = \sum_{n = 0}^{+ \infty} \frac{f^{(n)}(\zeta)}{n!} (z - \zeta)^{n}$ has universal Pad\'{e} approximants of Type I if for every compact set $K \subseteq \mathbb{C} \setminus \Omega$ with connected complement and for every function $h \in A(K)$, there exists a subsequence $(p_{k_n})_{n \in \mathbb{N}}$ of the sequence $(p_n)_{n \in \mathbb{N}}$ satisfying the following:

\begin{itemize}

\item[(i)]
$f \in D_{p_{k_n}, q_{k_n}}(\zeta)$ for every $n \in \mathbb{N}$.

\item[(ii)]
$\sup_{z \in K} |[f; p_{k_n} / q_{k_n}]_{\zeta}(z) - h(z)| \to 0$ as $n \to + \infty$.

\item[(iii)]
$\sup_{z \in J} |[f; p_{k_n} / q_{k_n}]_{\zeta}(z) - f(z)| \to 0$ as $n \to + \infty$, for every compact set $J \subseteq \Omega$.

\end{itemize}

Here the uniform convergences are meant with respect to the usual Euclidean distance in $\mathbb{C}$. The set of universal Pad\'{e} approximants of Type I is a dense and $G_{\delta}$ subset of $H(\Omega)$, where the space $H(\Omega)$ is endowed with the topology of uniform convergence on compacta. If $q_n = 0$, then this class of functions coincides with the class of universal Taylor series.

\bigskip

\textbf{Type II} (\cite{NESTORIDIS3}) Let $(p_n)_{n \in \mathbb{N}}, (q_n)_{n \in \mathbb{N}} \subseteq \mathbb{N}$ with $p_n, q_n \to + \infty$, $\Omega \subseteq \mathbb{C}$ be a domain and $\zeta \in \Omega$ be a fixed point. A holomorphic function $f \in H(\Omega)$ with Taylor expansion $f(z) = \sum_{n = 0}^{+ \infty} \frac{f^{(n)}(\zeta)}{n!} (z - \zeta)^{n}$ has universal Pad\'{e} approximants of Type II if for every compact set $K \subseteq \mathbb{C} \setminus \Omega$ with connected complement and for every rational function $h$, there exist two subsequence $(p_{k_n})_{n \in \mathbb{N}}$ and $(q_{k_n})_{n \in \mathbb{N}}$ of the sequences $(p_n)_{n \in \mathbb{N}}$ and $(q_n)_{n \in \mathbb{N}}$ respectively satisfying the following:

\begin{itemize}

\item[(i)]
$f \in D_{p_{k_n}, q_{k_n}}(\zeta)$ for every $n \in \mathbb{N}$.

\item[(ii)]
$\sup_{z \in K} \chi([f; p_{k_n} / q_{k_n}]_{\zeta}(z), h(z)) \to 0$ as $n \to + \infty$. The metric $\chi$ is the well - known distance defined on $\mathbb{C} \cup \{ \infty \}$.

\item[(iii)]
$\sup_{z \in J} |[f; p_{k_n} / q_{k_n}]_{\zeta}(z) - f(z)| \to 0$ as $n \to + \infty$ for every compact set $J \subseteq \Omega$.

\end{itemize}

The set of such functions is dense and $G_{\delta}$ subset of $H(\Omega)$, where the space $H(\Omega)$ is endowed with the topology of uniform convergence on compacta.

The universal Pad\'{e} approximants of Type I may be combined with universal Taylor series to yield simultaneous universal Pad\'{e} - Taylor approximation with the same indexes $p_{k_n}$ (\cite{MAKRIDIS1}, \cite{MAKRIDIS2}). Generally speaking, for any finite set of universal Pad\'{e} approximants of Type I with the same sequence $(p_n)_{n \in \mathbb{N}}$ we can prove simultaneous universal Pad\'{e} approximations with the same subsequence $(p_{k_n})_{n \in \mathbb{N}}$. In fact, we prove the following stronger theorem.

\begin{theorem} \label{theorem 1.1}

Let $(p_n)_{n \in \mathbb{N}} \subseteq \mathbb{N}$ with $p_n \to + \infty$. For every $n \in \mathbb{N}$, let $q^{(n)}_1, \cdots, q^{(n)}_{N(n)} \in \mathbb{N}$, where $N(n) \in \mathbb{N}$. Let $\Omega \subseteq \mathbb{C}$ be a simply connected domain and $\zeta \in \Omega$ be a fixed point. Then there exists a holomorphic function $f \in H(\Omega)$ with Taylor expansion $f(z) = \sum_{n = 0}^{+ \infty} \frac{f^{(n)}(\zeta)}{n!} (z - \zeta)^{n}$ such that for every compact set $K \subseteq \mathbb{C} \setminus \Omega$ with connected complement and for every function $h \in A(K)$, there exists a subsequence $(p_{k_n})_{n \in \mathbb{N}}$ of the sequence $(p_n)_{n \in \mathbb{N}}$ satisfying the following:

\begin{itemize}

\item[(i)]
$f \in D_{p_{k_n}, q^{(k_n)}_j}(\zeta)$ for every $j \in \{ 1, \cdots, N(k_n) \}$ and for every $n \in \mathbb{N}$.

\item[(ii)]
$\sup_{z \in K} |[f; p_{k_n} / q^{(k_n)}_{\sigma(k_n)}]_{\zeta}(z) - h(z)| \to 0$ as $n \to + \infty$ for every selection $\sigma: \mathbb{N} \to \mathbb{N}$ satisfying $\sigma(k_n) \in \{ 1, \cdots, N(k_n) \}$.

\item[(iii)]
$\sup_{z \in J} |[f; p_{k_n} / q^{(k_n)}_{\sigma(k_n)}]_{\zeta}(z) - f(z)| \to 0$ as $n \to + \infty$ for every compact set $J \subseteq \Omega$ and for every selection $\sigma: \mathbb{N} \to \mathbb{N}$ satisfying $\sigma(k_n) \in \{ 1, \cdots, N(k_n) \}$.

\end{itemize}

Moreover, the set of all such functions is dense and $G_{\delta}$ in $H(\Omega)$, where the space $H(\Omega)$ is endowed with the topology of uniform convergence on compacta.

\end{theorem}

Furthermore we obtain a similar result where we have one $q_n$ and finite many $p^{(n)}_1, \cdots, p^{(n)}_{N(n)} \in \mathbb{N}$ satisfying $\min \{ p^{(n)}_1, \cdots, p^{(n)}_{N(n)} \} \to + \infty$. See Theorem \ref{theorem 4.1} below.

These results are contained in section $3$. In section $4$ we prove the analogue results of Type II. We mention that we can obtain similar results for all centers of expansion simultaneously to be generic on a subspace of $A^{\infty}(\Omega)$ as well as of Seleznev type in the space of all formal power series. In section $5$ we modify the known construction of universal Taylor series and we conclude that the class $\mathcal{A}$ of all functions satisfying Theorem \ref{theorem 1.1} contains a dense and affine supspace of $H(\Omega)$. Further, we prove algebraic genericity for a class of functions slightly larger than the class $\mathcal{A}$. The difference is that instead of $(i)$ of Theorem \ref{theorem 1.1} we demand only that the functions $[f; p_{k_n} / q^{(k_n)}_{\sigma(k_n)}]_{\zeta}$ exist.

Finally section $2$ contains some preliminaries for Pad\'{e} approximants and the chordal distance.

\section{Preliminaries} 

\begin{definition} \label{definition 2.1}

Let $\zeta \in \mathbb{C}$ and $f$ a formal power series with center $\zeta$:
$$ f(z) = \sum_{k = 0}^{+ \infty} a_k (z - \zeta)^k. $$

Now, for every $p, q \in \mathbb{N}$ we consider a function of the form:
$$ [f; p / q]_{\zeta}(z) = \frac{A(z)}{B(z)} $$

where the functions $A(z)$ and $B(z)$ are polynomials such that $deg(A(z)) \leq p$, $deg(B(z)) \leq q$, $B(\zeta) = 1$ and also the Taylor expansion of the function $[f; p / q]_{\zeta}(z) = \sum_{k = 0}^{+ \infty} b_k (z - \zeta)^k$ (with center $\zeta \in \mathbb{C}$) satisfies:
$$ a_k = b_k \; \text{for every} \; k = 0, \cdots p + q. $$

If such a rational function exists, it is called the $(p, q)$ - Pad\'{e} approximant of $f$. Very often the power series $f$ is the Taylor development of a holomorphic function with center $\zeta$; a point in its domain of definition.

\end{definition}

\begin{remark} \label{remark 2.2}

According to Definition \ref{definition 2.1} we obtain that for $q = 0$ the $(p, 0)$ - Pad\'{e} approximant of $f$ exists trivially for every $p \in \mathbb{N}$, since:
$$ [f; p / 0]_{\zeta}(z) = \sum_{k = 0}^{p} a_k (z - \zeta)^k $$

for every $z \in \mathbb{C}$.
\end{remark}

\begin{remark} \label{remark 2.3}

For $q \geq 1$ Definition \ref{definition 2.1} does not necessarily implies the existence of Pad\'{e} approximants. However, if a Pad\'{e} approximant exists then it is unique as a rational funtion. It is known (\cite{BAKER.GRAVES - MORRIS}) that a necessary and sufficient condition for the existence and uniqueness of the polynomials $A(z)$ and $B(z)$ above is that the following $q \times q$ Hankel determinant:

$$ D_{p, q}(f, \zeta) = det
\begin{vmatrix}
a_{p - q + 1} & a_{p - q + 2} & \cdots & a_{p} \\ 
a_{p - q + 2} & a_{p - q + 3} & \cdots & a_{p + 1} \\
\vdots & \vdots & \ddots &\vdots \\
a_{p} & a_{p + 1} & \cdots & a_{p + q + 1} \\ 
\end{vmatrix} $$

is not equal to $0$, i.e. $D_{p, q}(f, \zeta)  \neq 0$. In the previous determinant we set $a_k = 0$ for every $k < 0$. In addition, if $D_{p, q}(f, \zeta)  \neq 0$ we also write $f \in D_{p, q}(\zeta)$.

In this case, the $(p, q)$ - Pad\'{e} approximant of $f$ (with center $\zeta \in \mathbb{C}$) is given by the following formula:
$$ [f; p / q]_{\zeta}(z) = \frac{A(f, \zeta)(z)}{B(f, \zeta)(z)} $$

where
$$ A(f, \zeta)(z) = $$
$$ = det
\begin{vmatrix}
(z - \zeta)^{q} S_{p - q}(f, \zeta)(z) & (z - \zeta)^{q - 1} S_{p - q + 1}(f, \zeta)(z) & \cdots & S_{p}(f, \zeta)(z) \\ 
a_{p - q + 1} & a_{p - q + 2} & \cdots & a_{p + 1} \\
\vdots & \vdots & \ddots &\vdots \\
a_{p} & a_{p + 1} & \cdots & a_{p + q} \\ 
\end{vmatrix} $$

\bigskip

and
$$ B(f, \zeta)(z) = det
\begin{vmatrix}
(z - \zeta)^{q} & (z - \zeta)^{q - 1} & \cdots & 1 \\ 
a_{p - q + 1} & a_{p - q + 2} & \cdots & a_{p + 1} \\
\vdots & \vdots & \ddots &\vdots \\
a_{p} & a_{p + 1} & \cdots & a_{p + q} \\ 
\end{vmatrix} $$

\bigskip

with
$$ S_k(f, \zeta)(z) = \begin{cases}
\sum_{i = 0}^{k} a_i (z - \zeta)^i &\mbox{if} \; k \geq 0 \\
0 &\mbox{if} \; k < 0 
\end{cases}. $$

The previous relations are called Jacobi formulas. Also, in this case, we notice that the polynomials $A(f, \zeta)(z)$ and $B(f, \zeta)(z)$ do not have any common zeros in $\mathbb{C}$ provided that $f \in D_{p, q}(\zeta)$.

\end{remark}

We will also make use of the following proposition.

\begin{proposition} \label{proposition 2.4} (\cite{BAKER.GRAVES - MORRIS}) Let $f(z) = \frac{A(z)}{B(z)}$ be a rational function where the functions $A(z)$ and $B(z)$ are polynomials with $deg(A(z)) = p_0$ and $deg(B(z))$ $= q_0$. In addition, suppose that $A(z)$ and $B(z)$ do not have any common zero in $\mathbb{C}$. Then for every $\zeta \in \mathbb{C}$ such that $B(\zeta) \neq 0$ we have:

\begin{itemize}

\item[(i)]
$f \in D_{p_0, q_0} (\zeta)$.

\item[(ii)]
$f \in D_{p, q_0} (\zeta)$ for every $p \geq p_0$.

\item[(iii)]
$f \in D_{p_0, q} (\zeta)$ for every $q \geq q_0$.

\end{itemize}

Moreover, for every $(p, q) \in \mathbb{N} \times \mathbb{N}$ with $p > p_0$ and $q > q_0$ we have:
$$ f \not\in D_{p, q} (\zeta). $$

In all cases above we obtain that $f(z) \equiv [f; p / q]_{\zeta}(z)$.

\end{proposition}

In some of the following results we make use of the chordal metric $\chi$, a metric defined on $\mathbb{C} \cup \{ \infty \}$. This is a well - known metric, given by the following relations:

For every $a, b \in \mathbb{C} \cup \{ \infty \}$, it holds:
$$ \chi(a, b) = \begin{cases} \frac{|a - b|}{\sqrt{1 + a^2} \cdot \sqrt{1 + b^2}} &\mbox{if} \; a, b \in \mathbb{C} 
\\ 
\frac{1}{\sqrt{1 + a^2}} &\mbox{if} \; a \in \mathbb{C} \; \text{and} \; b = \infty \\ 

0 &\mbox{if} \;a = b = \infty
 \end{cases} $$

We notice that for every $a, b \in \mathbb{C}$ it holds $\chi(a, b) \leq |a - b|$. In addition, $\chi(a, b) = \chi(\frac{1}{a}, \frac{1}{b})$ for every $a, b \in \mathbb{C} \cup \{ \infty \}$.

If $K \subseteq \mathbb{C}$ is a set and a sequence of functions $\{ f_n \}_{n \geq 1}: K \to \mathbb{C}$ converges uniformly to a function $f: K \to \mathbb{C}$ with respect to the Euclidean metric $|\cdot|$, so does with respect to $\chi$. Moreover, the metrics $|\cdot|$ and $\chi$ are uniformly equivalent on every compact subset $K$ of $\mathbb{C}$. Thus, if $\{ f_n \}_{n \geq 1}: E \to K$ converges $\chi$ - uniformly on $E$ towards a function $f: E \to K$, then automatically the convergence is also uniform with respect to $|\cdot|$.

\section{Simultaneous approximation for Universal Pad\'{e} approximants of Type I} 

We will need the following well - known lemmas.

\begin{lemma} \label{lemma 3.1} (\cite{MELAS.NESTORIDIS}, \cite{NESTORIDIS2}, \cite{NESTORIDIS4}) Let $\Omega$ be a domain in $\mathbb{C}$. Then there exists a sequence $\{ K_m \}_{m \geq 1}$ of compact subsets of $\mathbb{C} \setminus \Omega$ with connected complements, such that for every compact set $K \subseteq \mathbb{C} \setminus \Omega$ with connected complement, there exists an index $m \in \mathbb{N}$ satisfying $K \subseteq K_m$.

\end{lemma}

\begin{lemma} \label{lemma 3.2} (Existence of exhausting family; \cite{RUDIN}) Let $\Omega$ be an open set in $\mathbb{C}$. Then there exists a sequence $\{ L_k \}_{k \geq 1}$ of compact subsets of $\Omega$ such that:

\begin{itemize}

\item[(i)] 
$L_k \subseteq L_{k + 1}^o$ for every $k \in \mathbb{N}$.

\item[(ii)]
For every compact set $L \subseteq \Omega$ there exists an index $k \in \mathbb{N}$ such that $L \subseteq L_k$.

\item[(iii)]
Every connected component of $\tilde{\mathbb{C}} \setminus L_k$ contains at least one connected component of $\tilde{\mathbb{C}} \setminus \Omega$.
\end{itemize}
 
\end{lemma}

We present now the main theorem of this section.

\begin{theorem} \label{theorem 3.3}

Let $\Omega \subseteq \mathbb{C}$ be a simply connected domain and $L \subseteq \Omega$ be a compact set. We consider a sequence $(p_n)_{n \geq 1} \subseteq \mathbb{N}$ with $p_n \to + \infty$. Now, for every $n \in \mathbb{N}$ let $q_{1}^{(n)}, q_{2}^{(n)}, \cdots, q_{N(n)}^{(n)} \in \mathbb{N}$, where $N(n)$ is another natural number. Then there exists a function $f \in H(\Omega)$ satisfying the following:

For every compact set $K \subseteq \mathbb{C} \setminus \Omega$ with connected complement and for every function $h \in A(K)$, there exists a subsequence $(p_{k_n})_{n \geq 1}$ of the sequence $(p_n)_{n \geq 1}$ such that:

\begin{itemize}

\item[(1)]
$f \in D_{p_{k_n}, q_{j}^{(k_n)}}(\zeta)$ for every $\zeta \in L$, for every $n \in \mathbb{N}$ and for every $j \in \{ 1, \cdots, N(k_n) \}$.

\item[(2)]
$\max_{j = 1, \cdots, N(k_n)} \sup_{\zeta \in L} \sup_{z \in K} |[f; p_{k_n} / q_{j}^{(k_n)}]_{\zeta}(z) - h(z)| \to 0 \; \text{as} \; n \to + \infty$.

\item[(3)]
For every compact set $J \subseteq \Omega$ it holds:
$$ \max_{j = 1, \cdots, N(k_n)} \sup_{\zeta \in L} \sup_{z \in J} |[f; p_{k_n} / q_{j}^{(k_n)}]_{\zeta}(z) - f(z)| \to 0 \; \text{as} \; n \to + \infty. $$

\end{itemize}

Moreover, the set of all functions $f$ satisfing $(1) - (3)$ is dense and $G_{\delta}$ in $H(\Omega)$.

\end{theorem}

\begin{proof} Let $\{ f_i \}_{i \geq 1}$ be an enumeration of polynomials with coefficients in $\mathbb{Q} + i\mathbb{Q}$.

Now, for every $i, s, n, k, m \in \mathbb{N}$ and for every $j \in \{1, \cdots N(n) \}$ we consider the following sets:

$$ A(i, s, m, n, j) = \{ f \in H(\Omega) : f \in D_{p_{n}, q_{j}^{(n)}}(\zeta) \; \text{for every} \; \zeta \in L $$
$$ \text{and} \; \sup_{\zeta \in L} \sup_{z \in K_m} |[f; p_{n} / q_{j}^{(n)}]_{\zeta}(z) - f_i(z)| < \frac{1}{s}\}. $$ 

$$ A(i, s, m, n) = \{ f \in H(\Omega) : f \in D_{p_{n}, q_{j}^{(n)}}(\zeta) \; \text{for every} \; \zeta \in L $$
$$ \text{and for every} \; j = 1, 2, \cdots, N(n) \; \text{and also} $$
$$ \max_{j = 1, \cdots, N(n)} \sup_{\zeta \in L} \sup_{z \in K_m} |[f; p_{n} / q_{j}^{(n)}]_{\zeta}(z) - f_i(z)| < \frac{1}{s}\} \equiv \bigcap_{j = 1}^{N(n)} A(i, s, m, n, j). $$ 

$$ B(s, k, n, j) = \{ f \in H(\Omega) : f \in D_{p_{n}, q_{j}^{(n)}}(\zeta) \; \text{for every} \; \zeta \in L $$
$$ \text{and} \; \sup_{\zeta \in L} \sup_{z \in L_k} |[f; p_{n} / q_{j}^{(n)}]_{\zeta}(z) - f(z)| < \frac{1}{s}\}. $$ 

$$ B(s, k, n) = \{ f \in H(\Omega) : f \in D_{p_{n}, q_{j}^{(n)}}(\zeta) \; \text{for every} \; \zeta \in L $$
$$ \text{and for every} \; j = 1, 2, \cdots, N(n) \; \text{and also} $$
$$ \max_{j = 1, \cdots, N(n)} \sup_{\zeta \in L} \sup_{z \in L_k} |[f; p_{n} / q_{j}^{(n)}]_{\zeta}(z) - f(z)| < \frac{1}{s}\} \equiv \bigcap_{j = 1}^{N(n)} B(s, k, n, j). $$

One can verify (mainly by using Mergelyan's theorem) that if $\mathcal{U}$ is the set of all functions satisfying $(1) - (3)$, then it holds:
$$ \mathcal{U} = \bigcap_{i, s, k, m \in \mathbb{N}} \bigcup_{n \in \mathbb{N}} A(i, s, m, n) \cap B(s, k, n) = $$
$$ \bigcap_{i, s, k, m \in \mathbb{N}} \bigcup_{n \in \mathbb{N}} \bigcap_{j = 1}^{N(n)} A(i, s, m, n, j) \cap \bigcap_{j = 1}^{N(n)} B(s, k, n, j) $$

Since $H(\Omega)$ is a complete metric space, according to Baire's theorem, it suffices to prove the following:

\begin{claim} \label{claim 3.4} For every $i, s, n, k, m \in \mathbb{N}$ and for every $j \in \{1, \cdots N(n) \}$ the sets $A(i, s, m, n, j)$ and $B(s, k, n, j)$ are open in $H(\Omega)$.

\end{claim}

\begin{claim} \label{claim 3.5} For every $i, s, k, m \in \mathbb{N}$ the set $\mathcal{U}(i, s, k, m) = \bigcup_{n \in \mathbb{N}} A(i, s, m, n) \cap B(s, k, n) = \bigcup_{n \in \mathbb{N}} \bigcap_{j = 1}^{N(n)} A(i, s, m, n, j) \cap \bigcap_{j = 1}^{N(n)} B(s, k, n, j)$ is dense in $H(\Omega)$.

\end{claim}

\subsection{\boldmath{$A(i, s, m, n, j)$ and $B(s, k, n, j)$ are open in $H(\Omega)$}}

The sets $A(i, s, m, n, j)$ have been proven to be open in \cite{MELAS.NESTORIDIS} for $q_j^{(n)} = 0$ and in \cite{MAKRIDIS1} and \cite{MAKRIDIS2} for $q_j^{(n)} \geq 1$. The sets $B(s, k, n, j)$ have been proven to be open in \cite{FOURNODAULOS.NESTORIDIS} for $q_j^{(n)} \geq 1$ and in \cite{MELAS.NESTORIDIS} for $q_j^{(n)} = 0$. For complete proofs of the above facts we also refer to \cite{MAKRIDIS2}.

\subsection{\boldmath{Density of $\mathcal{U}(i, s, k, m)$}}

In order to prove Claim \ref{claim 3.5} we fix $i, s, k, m \in \mathbb{N}$ and we want to prove that the set:
$$ \mathcal{U}(i, s, k, m) = \bigcup_{n \in \mathbb{N}} A(i, s, m, n) \cap B(s, k, n) = $$
$$ = \bigcup_{n \in \mathbb{N}} \bigcap_{j = 1}^{N(n)} A(i, s, m, n, j) \cap \bigcap_{j = 1}^{N(n)} B(s, k, n, j) $$ 

is a dense subset of $H(\Omega)$.

Let $g \in H(\Omega)$, $L' \subseteq \Omega$ be a compact set and $\varepsilon > 0$. Our aim is to find a function $u \in \mathcal{U}(i, s, k, m)$ such that $\sup_{z \in L'} |u(z) - g(z)| < \varepsilon$. There is no problem if we assume that $\varepsilon < \frac{1}{s}$. According to Lemma \ref{lemma 3.2}, we are able to find an index $n_0 \in \mathbb{N}$ satisfying $L \cup L' \cup L_k \subseteq L_{n_0}$.

Since $L_{n_0}$ and $K_m$ are disjoint compact sets with connected complements, the set $L_{n_0} \cup K_m$ is also a compact set with connected complement.

Consider now the following function:
$$ w(z) = \begin{cases} f_i(z) &\mbox{if} \; z \in K_m \\ 
g(z) &\mbox{if} \; z \in L_{n_0} \end{cases}. $$

The function $w$ is well - defined (because $L_{n_0} \cap K_m = \emptyset$) and also $w \in A(L_{n_0} \cup K_m)$. We apply Mergelyan's theorem and so we find a polynomial $p$ such that $\sup_{L_{n_0} \cup K_m} |w(z) - p(z)| < \frac{\varepsilon}{2}$. Our assumption that $p_n \to + \infty$ allows us to find an index $p_{k_n} \in \mathbb{N}$ such that $p_{k_n} > deg(p(z))$.

Let $u(z) = p(z) + dz^{p_{k_n}}$, where $d \in \mathbb{C} \setminus \{ 0 \}$ and $|d| < \frac{\varepsilon}{2} \cdot \frac{1}{\sup_{z \in L_{n_0} \cup K_m} |z^{p_{k_n}}|}$. It is immediate that the function $u$ is a polynomial with $deg(u(z)) = p_{k_n}$. We also notice that it holds $\sup_{z \in L_{n_0} \cup K_m} |u(z) - w(z)| < \varepsilon$.

In order to complete the proof we have to verify that $u \in \mathcal{U}(i, s, k, m)$; we verify the following:

\begin{itemize}

\item[(i)]
$u \in A(i, s, m, k_n, j)$ for $j = 1, \cdots, N(k_n)$. Since $u$ is a polynomial with $deg(u(z)) = p_{k_n}$ we have that for every $\zeta \in L$ it holds $u \in D_{p_{k_n}, 0} (\zeta)$. It follows that $u \in D_{p_{k_n}, q_j^{(n)}} (\zeta)$ for every $j = 1, 2, \cdots, N(k_n)$. We also have:
$$ \max_{j = 1, \cdots, N(k_n)} \sup_{\zeta \in L} \sup_{z \in K_m} |[u; p_{n} / q_{j}^{(k_n)}]_{\zeta}(z) - f_i(z)| = $$
$$ = \sup_{\zeta \in L} \sup_{z \in K_m} |[u(z) - f_i(z)| < \frac{1}{s} $$

because $[u; p_{n} / q_{j}^{(k_n)}]_{\zeta}(z) = u(z)$ for every $\zeta \in L$ and for every $j = 1, \cdots, N(k_n)$, according to Proposition \ref{proposition 2.4}.

\item[(ii)]
$u \in B(s, k, k_n, j)$ for $j = 1, \cdots, N(k_n)$. Here we have only to verify that:
$$ \max_{j = 1, \cdots, N(k_n)} \sup_{\zeta \in L} \sup_{z \in L_k} |[u; p_{k_n} / q_{j}^{(k_n)}]_{\zeta}(z) - u(z)| = 0 < \frac{1}{s}. $$

This also holds because $[u; p_{k_n} / q_{j}^{(k_n)}]_{\zeta}(z) = u(z)$.

\end{itemize}

Since $u \in \bigcap_{j = 1}^{N(n)} A(i, s, m, n, j) \cap \bigcap_{j = 1}^{N(n)} B(s, k, n, j)$, it follows that $u \in \mathcal{U}(i, s, k, m)$. Baire's theorem yields the result.

\end{proof}

We present now two consequences of Theorem \ref{theorem 3.3}.

\begin{theorem} \label{theorem 3.6}

Let $\Omega \subseteq \mathbb{C}$ be a simply connected domain and $\zeta \in \Omega$ be a fixed point. We consider a sequence $(p_n)_{n \geq 1} \subseteq \mathbb{N}$ with $p_n \to + \infty$. Now, for every $n \in \mathbb{N}$ let $q_{1}^{(n)}, q_{2}^{(n)}, \cdots, q_{N(n)}^{(n)} \in \mathbb{N}$, where $N(n)$ is another natural number. Then there exists a function $f \in H(\Omega)$ satisfying the following:

For every compact set $K \subseteq \mathbb{C} \setminus \Omega$ with connected complement and for every function $h \in A(K)$, there exists a subsequence $(p_{k_n})_{n \geq 1}$ of the sequence $(p_n)_{n \geq 1}$ such that:

\begin{itemize}

\item[(1)]
$f \in D_{p_{k_n}, q_{j}^{(k_n)}}(\zeta)$ for every $n \in \mathbb{N}$ and for every $j \in \{ 1, \cdots, N(k_n) \}$.

\item[(2)]
$\max_{j = 1, \cdots, N(k_n)} \sup_{z \in K} |[f; p_{k_n} / q_{j}^{(k_n)}]_{\zeta}(z) - h(z)| \to 0 \; \text{as} \; n \to + \infty$.

\item[(3)]
For every compact set $J \subseteq \Omega$ it holds:
$$ \max_{j = 1, \cdots, N(k_n)} \sup_{z \in J} |[f; p_{k_n} / q_{j}^{(k_n)}]_{\zeta}(z) - f(z)| \to 0 \; \text{as} \; n \to + \infty. $$

\end{itemize}

Moreover, the set of all functions $f$ satisfing $(1) - (3)$ is dense and $G_{\delta}$ in $H(\Omega)$.

\end{theorem}

\begin{proof} We apply Theorem \ref{theorem 3.3} for $L = \{ \zeta \}$.

\end{proof}

\begin{theorem} \label{theorem 3.7}

Let $\Omega \subseteq \mathbb{C}$ be a simply connected domain. We consider a sequence $(p_n)_{n \geq 1} \subseteq \mathbb{N}$ with $p_n \to + \infty$. Now, for every $n \in \mathbb{N}$ let $q_{1}^{(n)}, q_{2}^{(n)}, \cdots, q_{N(n)}^{(n)} \in \mathbb{N}$, where $N(n)$ is another natural number. Then there exists a function $f \in H(\Omega)$ satisfying the following:

For every compact set $K \subseteq \mathbb{C} \setminus \Omega$ with connected complement and for every function $h \in A(K)$ there exists a subsequence $(p_{k_n})_{n \geq 1}$ of the sequence $(p_n)_{n \geq 1}$ such that:

\begin{itemize}

\item[(1)]
For every compact set $L \subseteq \Omega$, there exists a $n(L) \in \mathbb{N}$ such that $f \in D_{p_{k_n}, q_{j}^{(k_n)}}(\zeta)$ for every $\zeta \in L$, for every $n \geq n(L)$ (where $n(L)$ is a natural number depending on $L$) and for every $j \in \{ 1, \cdots, N(k_n) \}$.

\item[(2)]
$\max_{j = 1, \cdots, N(k_n)} \sup_{\zeta \in L} \sup_{z \in K} |[f; p_{k_n} / q_{j}^{(k_n)}]_{\zeta}(z) - h(z)| \to 0 \; \text{as} \; n \to + \infty$ for every compact set $L \subseteq \Omega$.

\item[(3)]
$\max_{j = 1, \cdots, N(k_n)} \sup_{\zeta \in L} \sup_{z \in L} |[f; p_{k_n} / q_{j}^{(k_n)}]_{\zeta}(z) - f(z)| \to 0 \; \text{as} \; n \to + \infty$ for every compact set $L \subseteq \Omega$.

\end{itemize}

Moreover, the set of all functions $f$ satisfing $(1) - (3)$ is dense and $G_{\delta}$ in $H(\Omega)$.

\end{theorem}

\begin{proof} Let $\mathcal{C}$ be the set of all functions satisfying Theorem \ref{theorem 3.7}. We apply Theorem \ref{theorem 3.3} for $L = L_k$ (and that for every $k \in \mathbb{N}$) and so we obtain a $G_{\delta}$ - dense class in $H(\Omega)$; the class $\mathcal{C}_k$. The reader is prompted to verify that $\mathcal{C} = \cap_{k \geq 1} \mathcal{C}_k$. The result follows from Baire's theorem. 

\end{proof}

We prove now a theorem analogue to Theorem \ref{theorem 3.3}, where the roles of $p$ and $q$ have been interchanged.

\begin{theorem} \label{theorem 3.8}

Let $\Omega \subseteq \mathbb{C}$ be a simply connected domain and $L \subseteq \Omega$ be a compact set. We consider an arbitrary sequence $(q_n)_{n \geq 1} \subseteq \mathbb{N}$ (may be bounded or unbounded). Now, for every $n \in \mathbb{N}$ let $p_{1}^{(n)}, p_{2}^{(n)}, \cdots, p_{N(n)}^{(n)} \in \mathbb{N}$, where $N(n)$ is another natural number, such that:
$$ \min_{j \in \{1, \cdots, N(n) \}} \{ p_{1}^{(n)}, p_{2}^{(n)}, \cdots, p_{N(n)}^{(n)} \} \to + \infty $$ 

as $n \to + \infty$. Then there exists a function $f \in H(\Omega)$ satisfying the following:

For every compact set $K \subseteq \mathbb{C} \setminus \Omega$ with connected complement and for every function $h \in A(K)$ there exists a subsequence $(q_{k_n})_{n \geq 1}$ of the sequence $(q_n)_{n \geq 1}$ such that:

\begin{itemize}

\item[(1)]
$f \in D_{p_{j}^{(k_n)}, q_{k_n}}(\zeta)$ for every $\zeta \in L$, for every $n \in \mathbb{N}$ and for every $j \in \{ 1, \cdots, N(k_n) \}$.

\item[(2)]
$\max_{j = 1, \cdots, N(k_n)} \sup_{\zeta \in L} \sup_{z \in K} |[f; p_{j}^{(k_n)} / q_{k_n}]_{\zeta}(z) - h(z)| \to 0 \; \text{as} \; n \to + \infty$.

\item[(3)]
For every compact set $J \subseteq \Omega$ it holds:

$$ \max_{j = 1, \cdots, N(k_n)} \sup_{\zeta \in L} \sup_{z \in J} |[f; p_{j}^{(k_n)} / q_{k_n}]_{\zeta}(z) - f(z)| \to 0 $$

$\text{as} \; n \to + \infty$.

\end{itemize}

Moreover, the set of all functions $f$ satisfing $(1) - (3)$ is dense and $G_{\delta}$ in $H(\Omega)$.

\end{theorem}

\begin{proof} Let $\{ f_i \}_{i \geq 1}$ be an enumeration of polynomials with coefficients in $\mathbb{Q} + i\mathbb{Q}$.

Now, for every $i, s, n, k, m \in \mathbb{N}$ and for every $j \in \{1, \cdots N(n) \}$ we consider the following sets:
$$ A(i, s, m, n, j) = \{ f \in H(\Omega) : f \in D_{p_{j}^{(n)}, q_{n}}(\zeta) \; \text{for every} \; \zeta \in L $$
$$ \text{and} \; \sup_{\zeta \in L} \sup_{z \in K_m} |[f; p_{j}^{(n)} / q_{n}]_{\zeta}(z) - f_i(z)| < \frac{1}{s}\}. $$ 

$$ A(i, s, m, n) = \{ f \in H(\Omega) : f \in D_{p_{j}^{(n)}, q_{n}}(\zeta) \; \text{for every} \; \zeta \in L $$
$$ \text{and for every} \; j = 1, 2, \cdots, N(n) \; \text{and also} $$
$$ \max_{j = 1, \cdots, N(n)} \sup_{\zeta \in L} \sup_{z \in K_m} |[f; p_{j}^{(n)} / q_{n}]_{\zeta}(z) - f_i(z)| < \frac{1}{s}\} \equiv \bigcap_{j = 1}^{N(n)} A(i, s, m, n, j). $$ 

$$ B(s, k, n, j) = \{ f \in H(\Omega) : f \in D_{p_{n}, q_{j}^{(n)}}(\zeta) \; \text{for every} \; \zeta \in L $$
$$ \text{and} \; \sup_{\zeta \in L} \sup_{z \in L_k} |[f; p_{j}^{(n)} / q_{n}]_{\zeta}(z) - f(z)| < \frac{1}{s}\}. $$ 

$$ B(s, k, n) = \{ f \in H(\Omega) : f \in D_{p_{j}^{(n)}, q_{n}}(\zeta) \; \text{for every} \; \zeta \in L $$
$$ \text{and for every} \; j = 1, 2, \cdots, N(n) \; \text{and also} $$
$$ \max_{j = 1, \cdots, N(n)} \sup_{\zeta \in L} \sup_{z \in L_k} |[f; p_{j}^{(n)} / q_{n}]_{\zeta}(z) - f(z)| < \frac{1}{s}\} \equiv \bigcap_{j = 1}^{N(n)} B(s, k, n, j). $$ 

One can verify (mainly by using Mergelyan's theorem) that if $\mathcal{U}$ is the set of all functions satisfying $(1) - (3)$, then it holds:
$$ \mathcal{U} = \bigcap_{i, s, k, m \in \mathbb{N}} \bigcup_{n \in \mathbb{N}} A(i, s, m, n) \cap B(s, k, n) = $$
$$ = \bigcap_{i, s, k, m \in \mathbb{N}} \bigcup_{n \in \mathbb{N}} \Big[ \bigcap_{j = 1}^{N(n)} A(i, s, m, n, j) \cap \bigcap_{j = 1}^{N(n)} B(s, k, n, j) \Big]. $$

Since $H(\Omega)$ is a complete metric space, according to Baire's theorem, it suffices to prove the following:

\begin{claim} \label{claim 3.9} For every $i, s, n, k, m \in \mathbb{N}$ and for every $j \in \{1, \cdots N(n) \}$ the sets $A(i, s, m, n, j)$ and $B(s, k, n, j)$ are open in $H(\Omega)$.

\end{claim}

\begin{claim} \label{claim 3.10} For every $i, s, k, m \in \mathbb{N}$ the set:
$$ \mathcal{U}(i, s, k, m) = \bigcup_{n \in \mathbb{N}} A(i, s, m, n) \cap B(s, k, n) = $$
$$ = \bigcup_{n \in \mathbb{N}} \bigcap_{j = 1}^{N(n)} A(i, s, m, n, j) \cap \bigcap_{j = 1}^{N(n)} B(s, k, n, j) $$ 

is dense in $H(\Omega)$.

\end{claim}

\subsection{\boldmath{$A(i, s, m, n, j)$ and $B(s, k, n, j)$ are open in $H(\Omega)$}}

The sets $A(i, s, m, n, j)$ have been proven to be open in \cite{MELAS.NESTORIDIS} for $q_j  = 0$ and in \cite{MAKRIDIS1} and \cite{MAKRIDIS2} for $q_j \geq 1$. The sets $B(s, k, n, j)$ have been proven to be open in \cite{FOURNODAULOS.NESTORIDIS} for $q_j \geq 1$ and in \cite{MELAS.NESTORIDIS} for $q_j = 0$. For complete proofs of the above facts we also refer to \cite{MAKRIDIS2}.

\subsection{\boldmath{Density of $\mathcal{U}(i, s, k, m)$}}

In order to prove Claim \ref{claim 3.10} we fix $i, s, k, m \in \mathbb{N}$ and we want to prove that the set:
$$ \mathcal{U}(i, s, k, m) = \bigcup_{n \in \mathbb{N}} A(i, s, m, n) \cap B(s, k, n) = $$
$$ = \bigcup_{n \in \mathbb{N}} \bigcap_{j = 1}^{N(n)} A(i, s, m, n, j) \cap \bigcap_{j = 1}^{N(n)} B(s, k, n, j) $$ 

is a dense subset of $H(\Omega)$.

Let $g \in H(\Omega)$, $L' \subseteq \Omega$ be a compact set and $\varepsilon > 0$. Our aim is to find a function $u \in \mathcal{U}(i, s, k, m)$ such that $\sup_{z \in L'} |u(z) - g(z)| < \varepsilon$. There is no problem if we assume that $\varepsilon < \frac{1}{s}$. According to Lemma \ref{lemma 3.2}, we are able to find an index $n_0 \in \mathbb{N}$ satisfying $L \cup L' \cup L_k \subseteq L_{n_0}$.

Since $L_{n_0}$ and $K_m$ are disjoint compact sets with connected complements, the set $L_{n_0} \cup K_m$ is also a compact set with connected complement.

Consider now the following function:
$$ w(z) = \begin{cases} f_i(z) &\mbox{if} \; z \in K_m \\ 
g(z) &\mbox{if} \; z \in L_{n_0} \end{cases} $$

The function $w$ is well defined (because $L_{n_0} \cap K_m = \emptyset$) and also $w \in A(L_{n_0} \cup K_m)$. We apply Mergelyan's theorem and we find a polynomial $p$ such that $\sup_{L_{n_0} \cup K_m} |w(z) - p(z)| < \frac{\varepsilon}{2}$. Since:
$$ \min_{j \in \{1, \cdots, N(n) \}} \{ p_{1}^{(n)}, p_{2}^{(n)}, \cdots, p_{N(n)}^{(n)} \} \to + \infty $$ 

as $n \to + \infty$, there exists an index $k_{n_1} \in \mathbb{N}$ such that:
$$ \min_{j \in \{1, \cdots, N(k_{n_1}) \}} \{ p_{1}^{(k_{n_1})}, p_{2}^{(k_{n_1})}, \cdots, p_{N(k_{n_1})}^{(k_{n_1})} \} > deg(p(z)). $$ 

Consider the function:
$$ u(z) = \frac{p(z)}{1 + dz^{q_{k_{n_1}}}} $$

where $d \in \mathbb{C} \setminus \{ 0 \}$ and $0 < |d|$ is small enough. We notice that it holds:
$$ \sup_{L_{n_0} \cup K_m} |p(z) - u(z)| = \sup_{L_{n_0} \cup K_m} |p(z) - \frac{p(z)}{1 + dz^{q_{k_{n_1}}}}| = \sup_{L_{n_0} \cup K_m} |\frac{dz^{q_{k_{n_1}}}p(z)}{1 + dz^{q_{k_{n_1}}}}| < \frac{\varepsilon}{2} $$

provided that $0 < |d|$ is small enough. It follows that $\sup_{L_{n_0} \cup K_m} |w(z) - u(z)| < \varepsilon$. In order to show that $u \in \mathcal{U}(i, s, k, m)$, we verify the following:

\begin{itemize}

\item[(1)]
$u \in D_{p_{j}^{(k_{n_1})}, q_{k_{n_1}}}(\zeta)$ for every $\zeta \in L_{n_0} \cup K$ and for every $j \in \{ 1, \cdots, N(k_{n_1}) \}$, according to Proposition \ref{proposition 2.4}. In particular, this holds for every $\zeta \in L$.

\item[(2)]
$[u; p_{j}^{(k_{n_1})} / q_{k_{n_1}}]_{\zeta}(z) = u(z)$ for every $\zeta \in L_{n_0} \cup K$, for every $z \in K$ and for every $j \in \{ 1, \cdots, N(k_{n_1}) \}$, according to Proposition \ref{proposition 2.4}. In particular, this holds for every $\zeta \in L$.

\end{itemize}

The result follows from Baire's theorem.

\end{proof}

We state now two consequences of Theorem \ref{theorem 3.8} without proofs.

\begin{theorem} \label{theorem 3.11}

Let $\Omega \subseteq \mathbb{C}$ be a simply connected domain and $\zeta \in \Omega$ a fixed point. We consider an arbitrary sequence $(q_n)_{n \geq 1} \subseteq \mathbb{N}$ (may be bounded or unbounded). Now, for every $n \in \mathbb{N}$ let $p_{1}^{(n)}, p_{2}^{(n)}, \cdots, p_{N(n)}^{(n)} \in \mathbb{N}$, where $N(n)$ is another natural number, such that:
$$ \min_{j \in \{1, \cdots, N(n) \}} \{ p_{1}^{(n)}, p_{2}^{(n)}, \cdots, p_{N(n)}^{(n)} \} \to + \infty $$ 

as $n \to + \infty$. Then there exists a function $f \in H(\Omega)$ satisfying the following:

For every compact set $K \subseteq \mathbb{C} \setminus \Omega$ with connected complement and for every function $h \in A(K)$, there exists a subsequence $(q_{k_n})_{n \geq 1}$ of the sequence $(q_n)_{n \geq 1}$ such that:

\begin{itemize}

\item[(1)]
$f \in D_{p_{j}^{(k_n)}, q_{k_n}}(\zeta)$ for every $n \in \mathbb{N}$ and for every $j \in \{ 1, \cdots, N(k_n) \}$.

\item[(2)]
$\max_{j = 1, \cdots, N(k_n)} \sup_{z \in K} |[f; p_{j}^{(k_n)} / q_{k_n}]_{\zeta}(z) - h(z)| \to 0 \; \text{as} \; n \to + \infty$ .

\item[(3)]
For every compact set $J \subseteq \Omega$ it holds:
$$ \max_{j = 1, \cdots, N(k_n)} \sup_{z \in J} |[f; p_{j}^{(k_n)} / q_{k_n}]_{\zeta}(z) - f(z)| \to 0 $$

$\text{as} \; n \to + \infty$.

\end{itemize}

Moreover, the set of all functions $f$ satisfing $(1) - (3)$ is dense and $G_{\delta}$ in $H(\Omega)$.

\end{theorem}

\begin{theorem} \label{theorem 3.12}

Let $\Omega \subseteq \mathbb{C}$ be a simply connected domain. We consider an arbitrary sequence $(q_n)_{n \geq 1} \subseteq \mathbb{N}$ (may be bounded or unbounded). Now, for every $n \in \mathbb{N}$ let $p_{1}^{(n)}, p_{2}^{(n)}, \cdots, p_{N(n)}^{(n)} \in \mathbb{N}$, where $N(n)$ is another natural number, such that:
$$ \min_{j \in \{1, \cdots, N(n) \}} \{ p_{1}^{(n)}, p_{2}^{(n)}, \cdots, p_{N(n)}^{(n)} \} \to + \infty $$ 

as $n \to + \infty$. Then there exists a function $f \in H(\Omega)$ such that for every compact set $K \subseteq \mathbb{C} \setminus \Omega$ with connected complement and for every function $h \in A(K)$, there exists a subsequence $(q_{k_n})_{n \geq 1}$ of the sequence $(q_n)_{n \geq 1}$ satisfying the following:

\begin{itemize}

\item[(1)]
For every compact set $L \subseteq \Omega$, there exists a $n(L) \in \mathbb{N}$ such that $f \in D_{p_{j}^{(k_n)}, q_{k_n}}(\zeta)$ for every $n \geq n(L)$, for every $j \in \{ 1, \cdots, N(k_n) \}$ and for every $\zeta \in L$.

\item[(2)]
$\max_{j = 1, \cdots, N(k_n)} \sup_{z \in K} |[f; p_{j}^{(k_n)} / q_{k_n}]_{\zeta}(z) - h(z)| \to 0 \; \text{as} \; n \to + \infty$.

\item[(3)]
For every compact set $L \subseteq \Omega$ it holds:
$$ \max_{j = 1, \cdots, N(k_n)} \sup_{z \in L} |[f; p_{j}^{(k_n)} / q_{k_n}]_{\zeta}(z) - f(z)| \to 0 $$

$\text{as} \; n \to + \infty$.

\end{itemize}

Moreover, the set of all functions $f$ satisfing $(1) - (3)$ is dense and $G_{\delta}$ in $H(\Omega)$.

\end{theorem}

We notice that the previous results have variants in several spaces with similar proofs. See for instance \cite{MAKRIDIS2}. We only mention the variants of Seleznev type and in a closed subspace of $A^{\infty}(\Omega)$. The proofs are similar to the ones of Theorem \ref{theorem 3.3} - Theorem \ref{theorem 3.7}.

Consider the space $\mathbb{C}^{\mathbb{N}}$ endowed with the Cartesian topology. A well - known result is that $\mathbb{C}^{\mathbb{N}}$ is a metrizable topological space; the same topology on $\mathbb{C}^{\mathbb{N}}$ can be induced by the following metric:
$$ \text{For every} \; a, b \in \mathbb{C}^{\mathbb{N}} \; \text{with} \; a \equiv (a_n)_{n \geq 0} \; \text{and} \; b \equiv (b_n)_{n \geq 0}  \; \text{we define}: $$
$$ \rho_c(a, b) = \sum_{n = 0}^{+ \infty} \frac{1}{2^n} \frac{|a_n - b_n|}{1 + |a_n - b_n|}. $$

We know that $(\mathbb{C}^{\mathbb{N}}, \rho_c)$ is a complete metric space.

Another metric that can be introduced on $\mathbb{C}^{\mathbb{N}}$ giving a different topology from the Cartesian one is the following:
$$ \text{For every} \; a, b \in \mathbb{C}^{\mathbb{N}} \; \text{with} \; a \equiv (a_n)_{n \geq 0} \; \text{and} \; b \equiv (b_n)_{n \geq 0}  \; \text{we define}: $$
$$ \rho_d(a, b) = 
\begin{cases} 
2^{- n_0} &\mbox{if} \; a \neq b \; (\text{where} \; n_0 = min \{ n \in \mathbb{N} : a_n \neq b_n \}) \\ 
0 & \mbox{if} \; a = b 
\end{cases} $$

It is also true that $(\mathbb{C}^{\mathbb{N}}, \rho_d)$ is a complete metric space. Moreover, one can see that $\rho_c \leq 2 \rho_d$.

\begin{theorem} \label{theorem 3.13} 

Let $(p_n)_{n \geq 0} \subseteq \mathbb{N}$ be a sequence such that $p_n \to + \infty$. Now, for every $n \in \mathbb{N}$ let $q_{0}^{(n)}, q_{1}^{(n)}, \cdots, q_{N(n)}^{(n)} \in \mathbb{N}$, where $N(n)$ is another natural number. Then there exists an element $a \equiv (a_n)_{n \geq 0} \in \mathbb{C}^{\mathbb{N}}$ such that the formal power series $f(z) = \sum_{n = 0}^{+ \infty} a_nz^n$ satisfies the following:

For every compact set $K \subseteq \mathbb{C} \setminus \{ 0 \}$ with connected complement and for every function $\psi \in A(K)$ there exists a subsequence $(p_{k_n})_{n \geq 0}$ of the sequence $(p_n)_{n \geq 0}$ such that:

\begin{itemize}

\item[(1)]
$f \in D_{p_{k_n}, q_{j}^{(k_n)}}(0)$ for every $j \in \{ 0, \cdots, N(k_n) \}$ and for every $n \in \mathbb{N}$.

\item[(2)]
$\max_{j = 0, \cdots, N(k_n)} \sup_{z \in K} |[f; p_{k_n} / q_{j}^{(k_n)}]_{0}(z) - \psi(z)| \to 0 \; \text{as} \; n \to + \infty$.

\end{itemize}

The set of all $a \equiv (a_n)_{n \geq 0} \in \mathbb{C}^{\mathbb{N}}$ satisfying $(1) - (2)$ is dense and $G_{\delta}$ in both spaces $(\mathbb{C}^{\mathbb{N}}, \rho_d)$ and $(\mathbb{C}^{\mathbb{N}}, \rho_c)$.

\end{theorem}

\begin{proof} Let $\{ f_i \}_{i \geq 1}$ be an enumeration of polynomials with coefficients in $\mathbb{Q} + i\mathbb{Q}$ and $\{ K_m \}_{m \geq 1}$ be a (fixed) absorbing family of $\mathbb{C} \setminus \{ 0 \}$. Now, for every $i, m, n, s \in \mathbb{N}$ and for every $j \in \{ 0, \cdots, N(n) \}$ we consider the following sets:
$$ X(i, j, m, n, s) = \{ a \equiv (a_0, a_1, \cdots) \in \mathbb{C}^\mathbb{N} : \; \text{the formal power series} $$
$$ f(z) = \sum_{n = 0}^{+ \infty} a_nz^n \; \text{satisfies} \; f \in D_{p_n, q_j^{(p_n)}}(0) $$
$$ \text{and} \; \sup_{z \in K_m} |[f; p_n / q_j^{(p_n)}]_{0}(z) - f_i(z)| < \frac{1}{s} \}. $$

$$ X(i, m, n, s) = \bigcap_{j = 0}^{N(n)} X(i, j, m, n, s) = $$
$$ = \{ a \equiv (a_0, a_1, \cdots) \in \mathbb{C}^\mathbb{N} : \; \text{the formal power series} f(z) = \sum_{n = 0}^{+ \infty} a_nz^n $$
$$ \text{satisfies} \; f \in D_{p_n, q_j^{(p_n)}}(0) \; \text{for every} \; j \in \{ 0, \cdots, N(n) \} \; \text{and} $$
$$ \sup_{z \in K_m} |[f; p_n / q_j^{(p_n)}]_{0}(z) - f_i(z)| < \frac{1}{s} \} \; \text{for every} \; j \in \{ 0, \cdots, N(n) \}. $$

One can verify that if $\mathcal{X}$ is the set of all functions satisfying $(1) - (2)$, then it holds:
$$ \mathcal{X} = \bigcap_{i, m, s} \bigcup_{n \in \mathbb{N}} X(i, m, n, s). $$

The sets $X(i, j, m, n, s)$ have been proven to be open in both spaces $(\mathbb{C}^{\mathbb{N}}, \rho_d)$ and $(\mathbb{C}^{\mathbb{N}}, \rho_c)$ in \cite{MAKRIDIS2}, thus the set $X(i, m, n, s)$ is also open as a finite intersection of the previous sets. On the other hand, the set $\mathcal{X}(i, m, s) = \bigcup_{n \in \mathbb{N}} X(i, m, n, s)$ is dense in $(\mathbb{C}^{\mathbb{N}}, \rho_d)$ (and so does in $(\mathbb{C}^{\mathbb{N}}, \rho_c)$ since $\rho_c \leq 2 \rho_d$); the proof is almost the same as the one of the corresponding part of Theorem $3.1$ in \cite{MAKRIDIS2} and is omitted. The result follows from Baire's theorem.

\end{proof}

Next we state some other results whose proofs are similar to the previous ones and are omitted.

\begin{theorem} \label{theorem 3.14} 

Let $(q_n)_{n \geq 0} \subseteq \mathbb{N}$ be an abitrary sequence (may be bounded or unbounded). Now, for every $n \in \mathbb{N}$ let $p_{0}^{(n)}, p_{1}^{(n)}, \cdots, p_{N(n)}^{(n)} \in \mathbb{N}$, where $N(n)$ is another natural number. Suppose that:
$$ \min_{j \in \{ 0, \cdots, N(n) \} } \{ p_{0}^{(n)}, \cdots, p_{N(n)}^{(n)} \} \to + \infty. $$

Then there exists an element $a \equiv (a_n)_{n \geq 0} \in \mathbb{C}^{\mathbb{N}}$ such that the formal power series $f(z) = \sum_{n = 0}^{+ \infty} a_nz^n$ satisfies the following:

For every compact set $K \subseteq \mathbb{C} \setminus \{ 0 \}$ with connected complement and for every function $\psi \in A(K)$ there exists a subsequence $(q_{k_n})_{n \geq 0}$ of $(q_n)_{n \geq 0}$ such that:

\begin{itemize}

\item[(1)]
$f \in D_{p_{j}^{(k_n)}, q_{k_n}}(0)$ for every $j \in \{ 0, \cdots, N(k_n) \}$ and for every $n \in \mathbb{N}$.

\item[(2)]
$\max_{j = 0, \cdots, N(k_n)} \sup_{z \in K} |[f; p_{j}^{(k_n)} / q_{k_n}]_{0}(z) - \psi(z)| \to 0 \; \text{as} \; n \to + \infty$. 

\end{itemize}

The set of all $a \equiv (a_n)_{n \geq 0} \in \mathbb{C}^{\mathbb{N}}$ satisfying $(1) - (2)$ is dense and $G_{\delta}$ in both spaces $(\mathbb{C}^{\mathbb{N}}, \rho_d)$ and $(\mathbb{C}^{\mathbb{N}}, \rho_c)$.

\end{theorem}

Let $\Omega \subseteq \mathbb{C}$ be an open set. We say that a holomorphic function $f$ defined on $\Omega$ belongs to $A^{\infty}(\Omega)$ if for every $l \in \mathbb{N}$ the $l$-th derivative $f^{(l)}$ of $f$ extends continuously on $\overline{\Omega}$. In $A^{\infty}(\Omega)$, we consider the topology defined by the seminorms $\sup_{z \in L_k} |f^{(l)}(z)|$, for every $k \geq 1$ and for every $l \in \mathbb{N}$, where $\{ L_k \}_{k \geq 1}$ is a family of compact subsets of $\overline{\Omega}$ such that for every compact set $L \subseteq \overline{\Omega}$ there exists an index $k \in \mathbb{N}$ satisfying $L \subseteq L_k$. Such a family for example is obtained by setting $L_k = \overline{\Omega} \cap \overline{B}(0, k)$ for every $k \in \mathbb{N}$. With this topology, $A^{\infty}(\Omega)$ becomes a Fr\'{e}chet space.

We call $X^{\infty}(\Omega)$ the closure in $A^{\infty}(\Omega)$ of all the rational functions with poles off $\overline{\Omega}$. If $\mathbb{C} \setminus \overline{\Omega}$ is connectet, then the polynomials are dense in $X^{\infty}(\Omega)$.

We state now without proof two generic results in $X^{\infty}(\Omega)$ analogue to the previous results.

\begin{theorem} \label{theorem 3.15}

Let $(p_n)_{n \geq 1} \subseteq \mathbb{N}$ be a sequence such that $p_n \to + \infty$. Now, for every $n \in \mathbb{N}$ let $q_{1}^{(n)}, q_{2}^{(n)}, \cdots, q_{N(n)}^{(n)} \in \mathbb{N}$, where $N(n)$ is another natural number.  Also, let $\Omega \subseteq \mathbb{C}$ be a domain such that $\mathbb{C} \setminus \overline{\Omega}$ is connected. Then there exists a function $f \in X^{\infty}(\Omega)$ satisfying the following:

For every compact set $K \subseteq \mathbb{C} \setminus \overline{\Omega}$ with connected complement and for every function $h \in A(K)$, there exists a subsequence $(p_{k_n})_{n \geq 1}$ of  the sequence $(p_n)_{n \geq 1}$ such that:

\begin{itemize}

\item[(1)]
For every compact set $L \subseteq \overline{\Omega}$, there exists a $n \equiv n(L) \in \mathbb{N}$ such that $f \in D_{p_{k_n}, q_{j}^{(k_n)}}(\zeta)$ for every $\zeta \in L$, for every $n \geq n(L)$ and for every $j \in \{ 1, \cdots, N(k_n) \}$.

\item[(2)]
For every $l \in \mathbb{N}$ it holds:
$$ \max_{j = 1, \cdots, N(k_n)} \sup_{\zeta \in L} \sup_{z \in L} |[f; p_{k_n} / q_{j}^{(k_n)}]^{(l)}_{\zeta}(z) - f^{(l)}(z)| \to 0 \; \text{as} \; n \to + \infty $$

for every compact set $L \subseteq \overline{\Omega}$.

\item[(3)]
$\max_{j = 1, \cdots, N(k_n)} \sup_{\zeta \in L} \sup_{z \in K} |[f; p_{k_n} / q_{j}^{(k_n)}]_{\zeta}(z) - h(z)|\to 0$ as $n \to + \infty$ for every compact set $L \subseteq \overline{\Omega}$.

\end{itemize}

The set of all functions satisfying $(1) - (3)$ is dense and $G_{\delta}$ in $X^{\infty}(\Omega)$. 

\end{theorem}

\begin{theorem} \label{theorem 3.16}

Let $(q_n)_{n \geq 1} \subseteq \mathbb{N}$ be an arbitrary sequence (may be bounded or unbounded). Now, for every $n \in \mathbb{N}$ let $p_{1}^{(n)}, p_{2}^{(n)}, \cdots, p_{N(n)}^{(n)} \in \mathbb{N}$, where $N(n)$ is another natural number. Suppose that:
$$ \min_{j \in \{ 0, \cdots, N(n) \} } \{ p_{0}^{(n)}, \cdots, p_{N(n)}^{(n)} \} \to + \infty. $$

Also, let $\Omega \subseteq \mathbb{C}$ be a domain such that $\mathbb{C} \setminus \overline{\Omega}$ is connected. Then there exists a function $f \in X^{\infty}(\Omega)$ so that the following holds:

For every compact set $K \subseteq \mathbb{C} \setminus \overline{\Omega}$ with connected complement and for every function $h \in A(K)$, there exists a subsequence $(q_{k_n})_{n \geq 1}$ of  the sequence $(q_n)_{n \geq 1}$ such that:
\begin{itemize}

\item[(1)]
For every compact set $L \subseteq \overline{\Omega}$, there exists a $n \equiv n(L) \in \mathbb{N}$ such that $f \in D_{p_{j}^{(k_n)}, q_{k_n}}(\zeta)$ for every $\zeta \in L$, for every $n \geq n(L)$ and for every $j \in \{ 1, \cdots, N(k_n) \}$.

\item[(2)]
For every $l \in \mathbb{N}$ it holds:
$$ \max_{j = 1, \cdots, N(k_n)} \sup_{\zeta \in L} \sup_{z \in L} |[f; p_{j}^{(k_n)} / q_{k_n}]^{(l)}_{\zeta}(z) - f^{(l)}(z)| \to 0 \; \text{as} \; n \to + \infty. $$

\item[(3)]
$\max_{j = 1, \cdots, N(k_n)} \sup_{\zeta \in L} \sup_{z \in K} |[f; p_{j}^{(k_n)} / q_{k_n}]_{\zeta}(z) - h(z)|\to 0$ as $n \to + \infty$.

\end{itemize}

Moreover, the set of all functions satisfying $(1) - (3)$ is dense and $G_{\delta}$ in $X^{\infty}(\Omega)$. 

\end{theorem}

\section{Simultaneous approximation for universal Pad\'{e} approximants of Type II}

In this section we prove for universal Pad\'{e} approximants of type II results as in the previous section and thus we strengthen results from \cite{NESTORIDIS3}. 

\begin{theorem} \label{theorem 4.1}

Let $\Omega \subseteq \mathbb{C}$ be an open set and $L, L' \subseteq \Omega$ two compact sets. Let also $K \subseteq \mathbb{C} \setminus \Omega$ be another compact set. We consider a sequence $(p_n)_{n \geq 1} \subseteq \mathbb{N}$ with $p_n \to + \infty$. Now, for every $n \in \mathbb{N}$ let $q_{1}^{(n)}, q_{2}^{(n)}, \cdots, q_{N(n)}^{(n)} \in \mathbb{N}$, where $N(n)$ is another natural number. Suppose also that:
$$ \min \{ q_{1}^{(n)}, q_{2}^{(n)}, \cdots, q_{N(n)}^{(n)} \} \to + \infty. $$ 

Then there exists a function $f \in H(\Omega)$ such that for every rational function $h$, there exists a subsequence $(p_{k_n})_{n \geq 1}$ of the sequence $(p_n)_{n \geq 1}$ satisfying the following:

\begin{itemize}

\item[(1)]
$f \in D_{p_{k_n}, q_{j}^{(k_n)}}(\zeta)$ for every $\zeta \in L$, for every $n \in \mathbb{N}$ and for every $j \in \{ 1, \cdots, N(k_n) \}$.

\item[(2)]
$\max_{j = 1, \cdots, N(k_n)} \sup_{\zeta \in L} \sup_{z \in K} \chi([f; p_{k_n} / q_{j}^{(k_n)}]_{\zeta}(z), h(z)) \to 0 \; \text{as} \; n \to + \infty$ .

\item[(3)]
$\max_{j = 1, \cdots, N(k_n)} \sup_{\zeta \in L} \sup_{z \in L'} |[f; p_{k_n} / q_{j}^{(k_n)}]_{\zeta}(z) - f(z)| \to 0 \; \text{as} \; n \to + \infty$.

\end{itemize}

Moreover, the set of all functions $f$ satisfing $(1) - (3)$ is dense and $G_{\delta}$ in $H(\Omega)$.

\end{theorem}

\begin{proof} Let $\{ f_i \}_{i \geq 1}$ be an enumeration of all rational functions with the coefficients of the numerator and denominator in $\mathbb{Q} + i\mathbb{Q}$. There is also no problem to assume that for every $i \geq 1$, the numerator and the denominator do not have any common zeros in $\mathbb{C}$ .

Now, for every $i, s, n \in \mathbb{N}$ and for every $j \in \{1, \cdots N(n) \}$ we consider the following sets:
$$ A(j, n, s) = \{ f \in H(\Omega) : f \in D_{p_{n}, q_{j}^{(n)}}(\zeta) \; \text{and} $$
$$ \sup_{\zeta \in L} \sup_{z \in L'} |[f; p_{n} / q_{j}^{(n)}]_{\zeta}(z) - f(z)| < \frac{1}{s}\}. $$ 

$$ A(n, s) = \{ f \in H(\Omega) : f \in D_{p_{n}, q_{j}^{(n)}}(\zeta) \; \text{for every} \; j = 1, 2, \cdots, N(n) $$
$$ \text{and} \; \max_{j = 1, \cdots, N(n)} \sup_{\zeta \in L} \sup_{z \in L'} |[f; p_{n} / q_{j}^{(n)}]_{\zeta}(z) - f(z)| < \frac{1}{s}\} \equiv \bigcap_{j = 1}^{N(n)} A(j, n, s). $$ 

$$ B(i, j, n, s) = \{ f \in H(\Omega) : f \in D_{p_{n}, q_{j}^{(n)}}(\zeta) \; \text{and} $$
$$ \sup_{\zeta \in L} \sup_{z \in K} \chi([f; p_{n} / q_{j}^{(n)}]_{\zeta}(z), f_i(z)) < \frac{1}{s} \}. $$ 

$$ B(i, n, s) = \{ f \in H(\Omega) : f \in D_{p_{n}, q_{j}^{(n)}}(\zeta) \; \text{for every} \; j = 1, 2, \cdots, N(n) $$
$$ \text{and} \; \max_{j = 1, \cdots, N(n)} \sup_{\zeta \in L} \sup_{z \in K} \chi([f; p_{n} / q_{j}^{(n)}]_{\zeta}(z), f_i(z)) < \frac{1}{s} \}  \equiv \bigcap_{j = 1}^{N(n)} B(i, j, n, s). $$ 

One can verify that if $\mathcal{U}$ is the set of all functions satisfying $(1) - (3)$, then it holds:
$$ \mathcal{U} = \bigcap_{i, s \in \mathbb{N}} \bigcup_{n \in \mathbb{N}} A(n, s) \cap B(i, n, s) = $$
$$ = \bigcap_{i, s \in \mathbb{N}} \bigcup_{n \in \mathbb{N}} \Big[ \bigcap_{j = 1}^{N(n)} A(j, n, s) \cap \bigcap_{j = 1}^{N(n)} B(i, j, n, s) \Big]. $$

The sets $A(j, n, s)$ and $B(i, j, n, s)$ have been proven to be open for every parameter in \cite{NESTORIDIS3}. Thus, the sets $A(n, s) \equiv \bigcap_{j = 1}^{N(n)} A(j, n, s)$ and $B(i, n, s) \equiv \bigcap_{j = 1}^{N(n)} B(i, j, n, s)$ are also open, as a finite intersection of open sets. It follows that $\mathcal{U}$ is a $G_{\delta}$ subset of $H(\Omega)$.

\subsection{Density of \boldmath{$\mathcal{U}(i, s)$}}

In order to use Baire's theorem, we fix the parameters $i, s \in \mathbb{N}$ and we want to prove that the set:
$$ \mathcal{U}(i, s) = \bigcup_{n \in \mathbb{N}} A(n, s) \cap B(i, n, s) $$

is dense in $H(\Omega)$.

Let $L'' \subseteq \Omega$ be a compact set, $\phi \in H(\Omega)$ and $\varepsilon > 0$. We assume that $\varepsilon < \frac{1}{s}$. Our aim is to find a function $g \in \mathcal{U}(i, s)$ such that:
$$ \sup_{z \in L''} |\phi(z) - g(z)| < \varepsilon. $$

Without loss of generality, we suppose that $L \cup L' \subseteq (L'')^o$ and also that every connected component of $\mathbb{C} \cup \{ \infty \} \setminus L''$ contains a connected component of $\mathbb{C} \cup \{ \infty \} \setminus \Omega$. For instance, that can be achieved by using Lemma \ref{lemma 3.2}.

Consider now the following function:
$$ w(z) = \begin{cases} f_i(z) &\mbox{if} \; z \in K \\ 
\phi(z) &\mbox{if} \; z \in L'' \end{cases}. $$

The set of poles of $f_i$ on $K$ is finite; let $\mu$ denote the sum of the principal parts of $f_i$ on these poles. So, the function $\omega - \mu$ is holomorphic in an open set containing $L'' \cup K$. We apply Runge's theorem to approximate the function $\omega - \mu$ uniformly on $L'' \cup K$ with respect to the Euclidean distance by a sequence of rational functions: 
$$ \frac{\tilde{A}_n(z)}{\tilde{B}_n(z)} $$ 

So, there exists a natural number $n_0 \in \mathbb{N}$ satisfying the following:
$$ \sup_{z \in L'' \cup K} |(\omega(z) - \mu(z)) - \frac{\tilde{A}_n(z)}{\tilde{B}_n(z)}| < \frac{\varepsilon}{2} \; \text{for every} \; n \geq n_0. $$

In paricular, $\tilde{B}_n(z) \neq 0$ for every $z \in L'' \cup K$ and for every $n \geq n_0$.
There is also no problem to assume that the polynomials $\tilde{A}_n(z)$ and $\tilde{B}_n(z)$ have no common zeros in $\mathbb{C}$. On the other hand, the sequence of functions:

$$ \mu(z) + \frac{\tilde{A}_n(z)}{\tilde{B}_n(z)} = \frac{A_n(z)}{B_n(z)} $$ 

defined for $n \geq n_0$, satisfies:
$$ \sup_{z \in K} \chi(f_i(z), \mu(z)) + \frac{\tilde{A}_n(z)}{\tilde{B}_n(z)}) < \frac{\varepsilon}{2} $$

and 
$$ \sup_{z \in L''} |\phi(z) - \mu(z)) - \frac{\tilde{A}_n(z)}{\tilde{B}_n(z)}| < \frac{\varepsilon}{2} $$

for every $n \geq n_0$.

We notice that the polynomials $A_n(z)$ and $B_n(z)$ have no common zeros in $\mathbb{C}$, since:
$$ \mu(z) + \frac{\tilde{A}_n(z)}{\tilde{B}_n(z)} = \frac{\mu(z) + \tilde{A}_n(z)}{\tilde{B}_n(z)} = \frac{A_n(z)}{B_n(z)} $$ 

and also that $B_n(z) \neq 0$ for every $z \in L''$ and for every $n \geq n_0$, because the polynomials $\tilde{B}_n(z)$ have the same property for every $n \geq n_0$.

Since $p_n \to + \infty$ and $\min \{ q_{1}^{(n)}, q_{2}^{(n)}, \cdots, q_{N(n)}^{(n)} \} \to + \infty$, there exists an index $k_{n_0} > n_0\in \mathbb{N}$ such that:
$$ p_{k_{n_0}} > \max{ \{ deg(A_{n_0}(z)), deg(B_{n_0}(z)) \} } $$

and
$$ \min \{ q_{1}^{(k_{n_0})}, q_{2}^{(k_{n_0})}, \cdots, q_{N(k_{n_0})}^{(k_{n_0})} \} > deg(B_{n_0}(z)). $$

We set $t = p_{k_{n_0}} - deg(B_{n_0}(z))$ and we consider the function:
$$ \frac{A_{n_0}(z)}{B_{n_0}(z)} + dz^t = \frac{A_{n_0}(z) + dz^tB_{n_0}(z)}{B_{n_0}(z)}. $$

Now, for every $d \in \mathbb{C}$, the polynomials $A_{n_0}(z) + dz^tB_{n_0}(z)$ and $B_{n_0}(z)$ do not have any common zero in $\mathbb{C}$, because the polynomials $A_{n_0}(z)$ and $B_{n_0}(z)$ do so. If the parameter $d \in \mathbb{C}$ is close to zero (it suffices to demand $d \cdot \sup_{z \in L'' \cup K} |z^t| < \frac{\varepsilon}{2}$), then:
$$ \sup_{z \in K} \chi(\frac{A_{n_0}(z)}{B_{n_0}(z)} + dz^t, \omega(z)) < \varepsilon $$

and
$$ \sup_{z \in L''} |\frac{A_{n_0}(z)}{B_{n_0}(z)} + dz^t - \omega(z)| < \varepsilon. $$

Since $deg(B_{n_0}(z)) < \min \{ q_{1}^{(k_{n_0})}, q_{2}^{(k_{n_0})}, \cdots, q_{N(k_{n_0})}^{(k_{n_0})} \}$, for $d \in \mathbb{C} \setminus \{ 0 \}$ we have the following:

\begin{itemize}

\item[(1)] 
$deg(A_{n_0}(z) + dz^tB_{n_0}(z)) = p_{k_{n_0}}$.

\item[(2)]
 According to Proposition \ref{proposition 2.4} it holds:
$$ \frac{A_{n_0}(z)}{B_{n_0}(z)} + dz^t = \frac{A_{n_0}(z) + dz^tB_{n_0}(z)}{B_{n_0}(z)} \in D_{p_{k_{n_0}}, q_{j}^{(k_{n_0})}}(\zeta) $$ 

for every $j \in \{ 1, 2, \cdots, N(k_{n_0}) \}$ and for every $\zeta \in \mathbb{C}$ such that $B_{n_0}(\zeta) \neq 0$; in particular this holds for every $\zeta \in L$.

\item[(3)]
$$ [\frac{A_{n_0}(z)}{B_{n_0}(z)} + dz^t; p_{(k_{n_0})} / q_{j}^{(k_{n_0})}]_{\zeta}(z) = \frac{A_{n_0}(z)}{B_{n_0}(z)} + dz^t $$ 

for every $j \in \{ 1, 2, \cdots, N(k_{n_0}) \}$ and for every $\zeta \in L$. 

\end{itemize}

Thus, it holds:
$$ \max_{j \in \{ 1, 2, \cdots, N(k_{n_0}) \} } \sup_{\zeta \in L} \sup_{z \in K} \chi([\frac{A_{n_0}(z)}{B_{n_0}(z)} + dz^t; p_{k_{n_0}} / q_{j}^{(k_{n_0})}]_{\zeta}(z), f_i(z)) < \varepsilon $$

and
$$ \max_{j \in \{ 1, 2, \cdots, N(k_{n_0}) \} } \sup_{\zeta \in L} \sup_{z \in L'} |[\frac{A_{n_0}(z)}{B_{n_0}(z)} + dz^t; p_{k_{n_0}} / q_{j}^{(k_{n_0})}]_{\zeta}(z) - (\frac{A_{n_0}(z)}{B_{n_0}(z)} + dz^t)| = $$
$$ = 0 < \frac{1}{s} < \varepsilon. $$

We also have that:
$$ \sup_{z \in L''} |\frac{A_{n_0}(z)}{B_{n_0}(z)} + dz^t - \phi(z)| < \varepsilon. $$

Since the polynomials $A_{n_0}(z) + dz^t$ and $B_{n_0}(z)$ have no common zeros, we have that:
$$ \min_{z \in L' \cup K} |A_{n_0}(z) + dz^tB_{n_0}(z)|^2 + |B_{n_0}(z)|^2 > 0. $$ 

One can verify that these polynomials are the ones given by the Jacobi formulas for the function:
$$ [\frac{A_{n_0}(z)}{B_{n_0}(z)} + dz^t; p_{k_{n_0}} / q_{j}^{(k_{n_0})}]_{\zeta}(z) $$

for any $\zeta \in \mathbb{C}$ with $B_{n_0}(\zeta) \neq 0$ and for every $j \in \{ 1, 2, \cdots, N(k_{n_0}) \}$.

Since every connected component of $\mathbb{C} \cup \{ \infty \} \setminus L''$ contains a connected component of $\mathbb{C} \cup \{ \infty \} \setminus \Omega$, every zero of $B_{n_0}(z)$ in $\Omega \setminus L''$ lies in the same connected componet of $\mathbb{C} \cup \{ \infty \} \setminus L''$ with a point in $\mathbb{C} \cup \{ \infty \} \setminus \Omega$. Therefore we may approximate the function:
$$ \frac{A_{n_0}(z)}{B_{n_0}(z)} + dz^t $$

by a function $g \in H(\Omega)$. The approximations is uniform on $L''$ with respect to the Euclidean distance. Since $L \subseteq (L'')^o$, there exists $r > 0$ such that:
$$ \{ z \in \mathbb{C} : |z - \zeta| \leq r \} \subseteq (L'')^o \; \text{for all} \; \zeta \in L. $$

Now Cauchy estimates allow us to show that a finite number of Taylor coefficients of $g$ with center $\zeta \in L$ are uniformly close to the corresponding coefficients of:
$$ \frac{A_{n_0}(z)}{B_{n_0}(z)} + dz^t. $$

It is now easy to see that $g$ satisfies all requirements; the only difference of $g$ from $\frac{A_{n_0}(z)}{B_{n_0}(z)} + dz^t$ is that it is not true that $[g;p_{k_{n_0}} / q_{j}^{(k_{n_0})}]_{\zeta}(z) = g(z)$, but instead it holds:
$$ \max_{j \in \{ 1, 2, \cdots, N(k_{n_0}) \} } \sup_{\zeta \in L} \sup_{z \in L'} |[g; p_{k_{n_0}} / q_{j}^{(k_{n_0})}]_{\zeta}(z) - g(z)| \leq $$
$$ \max_{j \in \{ 1, 2, \cdots, N(k_{n_0}) \} } \sup_{\zeta \in L} \sup_{z \in L'} |[g; p_{k_{n_0}} / q_{j}^{(k_{n_0})}]_{\zeta}(z) - [\frac{A_{n_0}(z)}{B_{n_0}(z)} + dz^t; p_{k_{n_0}} / q_{j}^{(k_{n_0})}]_{\zeta}(z)| + $$
$$ + \max_{j \in \{ 1, 2, \cdots, N(k_{n_0}) \} } \sup_{z \in L'} |[\frac{A(z)}{B(z)} + dz^t; p_{k_{n_0}} / q_{j}^{(k_{n_0})}]_{\zeta}(z) - g(z)| $$

where the last two terms are small enough because $p_{k_{n_0}}$ and $q_{j}^{(k_{n_0})}$ are already fixed and thus we know which set of Taylor coefficients we have to control. Baire's theorem completes the proof.

\end{proof}

We present now two consequences of Theorem \ref{theorem 4.1}.

\begin{theorem} \label{theorem 4.2}

Let $\Omega \subseteq \mathbb{C}$ be an open set and $\zeta \in \Omega$ be a fixed point. We consider a sequence $(p_n)_{n \geq 1} \subseteq \mathbb{N}$ with $p_n \to + \infty$. Now, for every $n \in \mathbb{N}$ let $q_{1}^{(n)}, q_{2}^{(n)}, \cdots, q_{N(n)}^{(n)} \in \mathbb{N}$, where $N(n)$ is another natural number. Suppose also that $\min \{ q_{1}^{(n)}, q_{2}^{(n)}, \cdots, q_{N(n)}^{(n)} \} \to + \infty$. Then there exists a function $f \in H(\Omega)$ such that for every compact set $K \subseteq \mathbb{C} \setminus \Omega$ and for every rational function $h$, there exists a subsequence $(p_{k_n})_{n \geq 1}$ of the sequence $(p_n)_{n \geq 1}$ satisfying the following:

\begin{itemize}

\item[(1)]
$f \in D_{p_{k_n}, q_{j}^{(k_n)}}(\zeta)$ for every $n \in \mathbb{N}$ and for every $j \in \{ 1, \cdots, N(k_n) \}$.

\item[(2)]
$\max_{j = 1, \cdots, N(k_n)} \sup_{z \in K} \chi([f; p_{k_n} / q_{j}^{(k_n)}]_{\zeta}(z), h(z)) \to 0 \; \text{as} \; n \to + \infty$.

\item[(3)]
$\max_{j = 1, \cdots, N(k_n)} \sup_{z \in L'} |[f; p_{k_n} / q_{j}^{(k_n)}]_{\zeta}(z) - f(z)| \to 0 \; \text{as} \; n \to + \infty$ for every compact set $L' \subseteq \Omega$.

\end{itemize}

Moreover, the set of all functions $f$ satisfing $(1) - (3)$ is dense and $G_{\delta}$ in $H(\Omega)$.

\end{theorem}

\begin{proof} It suffices to apply Theorem \ref{theorem 4.1} for $L = \{ \zeta \}$.

\end{proof}

\begin{theorem} \label{theorem 4.3}

Let $\Omega \subseteq \mathbb{C}$ be an open. We consider a sequence $(p_n)_{n \geq 1} \subseteq \mathbb{N}$ with $p_n \to + \infty$. Now, for every $n \in \mathbb{N}$ let $q_{1}^{(n)}, q_{2}^{(n)}, \cdots, q_{N(n)}^{(n)} \in \mathbb{N}$, where $N(n)$ is another natural number. Suppose that:
$$ \min \{ q_{1}^{(n)}, q_{2}^{(n)}, \cdots, q_{N(n)}^{(n)} \} \to + \infty.$$ 

Then there exists a function $f \in H(\Omega)$ such that for every compact set $K \subseteq \mathbb{C} \setminus \Omega$ and for every rational function $h$, there exists a subsequence $(p_{k_n})_{n \geq 1}$ of the sequence $(p_{n})_{n \geq 1}$ satisfying the following:

\begin{itemize}

\item[(1)]
For every compact set $L \subseteq \Omega$, there exists a $n(L) \in \mathbb{N}$ such that $f \in D_{p_{k_n}, q_{j}^{(k_n)}}(\zeta)$ for every $\zeta \in L$, for every $n \geq n(L)$ and for every $j \in \{ 1, \cdots, N(k_n) \}$.

\item[(2)]
$\max_{j = 1, \cdots, N(k_n)} \sup_{\zeta \in L} \sup_{z \in K} \chi([f; p_{k_n} / q_{j}^{(k_n)}]_{\zeta}(z), h(z)) \to 0 \; \text{as} \; n \to + \infty$ for every compact set $L \subseteq \Omega$.

\item[(3)]
$\max_{j = 1, \cdots, N(k_n)} \sup_{\zeta \in L} \sup_{z \in L} |[f; p_{k_n} / q_{j}^{(k_n)}]_{\zeta}(z) - f(z)| \to 0 \; \text{as} \; n \to + \infty$ for every compact set $L \subseteq \Omega$.

\end{itemize}

Moreover, the set of all functions $f$ satisfing $(1) - (3)$ is dense and $G_{\delta}$ in $H(\Omega)$.

\end{theorem}

\begin{proof} We apply Theorem \ref{theorem 4.1} for $L = L' = L_k$ and for $K = K_m$. In that way we obtain a $G_{\delta}$ - dense class $\mathcal{U}_{k, m}$ of $H(\Omega)$. One can verify (by using a diagonal argument) that if  $\mathcal{U}$ is the class of all functions of $H(\Omega)$ satisfying $(1) - (3)$, then it holds:
$$ \mathcal{U} = \bigcup_{n, m \in \mathbb{N}} \mathcal{U}_{k, m}. $$

The result follows once more from Baire's theorem.

\end{proof}

Now we present similar results where the roles of $p$ and $q$ have been interchanged.

\begin{theorem} \label{theorem 4.4}

Let $\Omega \subseteq \mathbb{C}$ be an open set and $L, L' \subseteq \Omega$ two compact sets. Let also $K \subseteq \mathbb{C} \setminus \Omega$ be another compact set. We consider a sequence $(q_n)_{n \geq 1} \subseteq \mathbb{N}$ with $q_n \to + \infty$. Now, for every $n \in \mathbb{N}$ let $p_{1}^{(n)}, p_{2}^{(n)}, \cdots, p_{N(n)}^{(n)} \in \mathbb{N}$, where $N(n)$ is another natural number. Suppose also that:
$$ \min \{ p_{1}^{(n)}, p_{2}^{(n)}, \cdots, p_{N(n)}^{(n)} \} \to + \infty. $$

Then there exists a function $f \in H(\Omega)$ such that for every rational function $h$, there exists a subsequence $(q_{k_n})_{n \geq 1}$ of the sequence $(q_n)_{n \geq 1}$ such that the following hold:

\begin{itemize}

\item[(1)]
$f \in D_{p_{j}^{(k_n)}, q_{k_n}}(\zeta)$ for every $\zeta \in L$, for every $n \in \mathbb{N}$ and for every $j \in \{ 1, \cdots, N(k_n) \}$.

\item[(2)]
$\max_{j = 1, \cdots, N(k_n)} \sup_{\zeta \in L} \sup_{z \in K} \chi([f; p_{j}^{(k_n)} / q_{k_n}]_{\zeta}(z), h(z)) \to 0 \; \text{as} \; n \to + \infty$.

\item[(3)]
$\max_{j = 1, \cdots, N(k_n)} \sup_{\zeta \in L} \sup_{z \in L'} |[f; p_{j}^{(k_n)} / q_{k_n}]_{\zeta}(z) - f(z)| \to 0 \; \text{as} \; n \to + \infty$.

\end{itemize}

Moreover, the set of all functions $f$ satisfing $(1) - (3)$ is dense and $G_{\delta}$ in $H(\Omega)$.

\end{theorem}

\begin{theorem} \label{theorem 4.5}

Let $\Omega \subseteq \mathbb{C}$ be an open set and $\zeta \in \Omega$ be a fixed point. We consider a sequence $(q_n)_{n \geq 1} \subseteq \mathbb{N}$ with $q_n \to + \infty$. Now, for every $n \in \mathbb{N}$ let $p_{1}^{(n)}, p_{2}^{(n)}, \cdots, p_{N(n)}^{(n)} \in \mathbb{N}$, where $N(n)$ is another natural number. Suppose also that $\min \{ p_{1}^{(n)}, p_{2}^{(n)}, \cdots, p_{N(n)}^{(n)} \} \to + \infty$. Then there exists a function $f \in H(\Omega)$ such that for every compact set $K \subseteq \mathbb{C} \setminus \Omega$ and for every rational function $h$, there exists a subsequence $(q_{k_n})_{n \geq 1}$ of the sequence $(q_n)_{n \geq 1}$ satisfying the following:

\begin{itemize}

\item[(1)]
$f \in D_{p_{j}^{(k_n)}, q_{k_n}}(\zeta)$ for every $n \in \mathbb{N}$ and for every $j \in \{ 1, \cdots, N(k_n) \}$.

\item[(2)]
$\max_{j = 1, \cdots, N(k_n)} \sup_{z \in K} \chi([f; p_{j}^{(k_n)} / q_{k_n}]_{\zeta}(z), h(z)) \to 0 \; \text{as} \; n \to + \infty$.

\item[(3)]
$\max_{j = 1, \cdots, N(k_n)} \sup_{z \in L'} |[f; p_{j}^{(k_n)} / q_{k_n}]_{\zeta}(z) - f(z)| \to 0 \; \text{as} \; n \to + \infty$ for every compact set $L' \subseteq \Omega$.

\end{itemize}

Moreover, the set of all functions $f$ satisfing $(1) - (3)$ is dense and $G_{\delta}$ in $H(\Omega)$.

\end{theorem}

\begin{proof} It suffices to apply Theorem \ref{theorem 4.4} for $L = \{ \zeta \}$.

\end{proof}

\begin{theorem} \label{theorem 4.6}

Let $\Omega \subseteq \mathbb{C}$ be an open set. We consider a sequence $(q_n)_{n \geq 1} \subseteq \mathbb{N}$ with $q_n \to + \infty$. Now, for every $n \in \mathbb{N}$ let $p_{1}^{(n)}, p_{2}^{(n)}, \cdots, p_{N(n)}^{(n)} \in \mathbb{N}$, where $N(n)$ is another natural number. Suppose that:
$$ \min \{ p_{1}^{(n)}, p_{2}^{(n)}, \cdots, p_{N(n)}^{(n)} \} \to + \infty. $$

Then there exists a function $f \in H(\Omega)$ such that for every compact set $K \subseteq \mathbb{C} \setminus \Omega$ and for every rational function $h$, there exists a subsequence $(q_{k_n})_{n \geq 1}$ of the sequence $(q_n)_{n \geq 1}$ satisfying the following:

\begin{itemize}

\item[(1)]
For every compact set $L \subseteq \Omega$, there exists a $n(L) \in \mathbb{N}$ such that $f \in D_{p_{j}^{(k_n)}, q_{k_n}}(\zeta)$ for every $n \geq n(L)$ and for every $j \in \{ 1, \cdots, N(k_n) \}$.

\item[(2)]
$\max_{j = 1, \cdots, N(k_n)} \sup_{\zeta \in L} \sup_{z \in K} \chi([f; p_{j}^{(k_n)} / q_{k_n}]_{\zeta}(z), h(z)) \to 0 \; \text{as} \; n \to + \infty$ for every compact set $L \subseteq \Omega$.

\item[(3)]
$\max_{j = 1, \cdots, N(k_n)} \sup_{\zeta \in L} \sup_{\zeta \in L} |[f; p_{j}^{(k_n)} / q_{k_n}]_{\zeta}(z) - f(z)| \to 0 \; \text{as} \; n \to + \infty$ for every compact set $L \subseteq \Omega$.

\end{itemize}

Moreover, the set of all functions $f$ satisfing $(1) - (3)$ is dense and $G_{\delta}$ in $H(\Omega)$.

\end{theorem}

\begin{proof} We apply Theorem \ref{theorem 4.4} for $L = L' = L_k$ and for $K = K_m$. In that way we obtain a $G_{\delta}$ - dense class $\mathcal{U}_{k, m}$ of $H(\Omega)$. One can verify (by using a diagonal argument) that if  $\mathcal{U}$ is the class of all functions of $H(\Omega)$ satisfying $(1) - (3)$, then it holds:
$$ \mathcal{U} = \bigcup_{n, m \in \mathbb{N}} \mathcal{U}_{k, m}. $$

The result follows once more from Baire's theorem.

\end{proof}

Next we strengthen a result from \cite{NESTORIDIS3} conserning meromorphic functions.

Let $\Omega \subseteq \mathbb{C}$ be an open set. We consider the following set:

$$ Z(\Omega) = \{ f: \Omega \to \mathbb{C} \cup \{ \infty \} \} $$

endowed with the topology of uniform convergence with respect to $\chi$ on each compact subset of $\Omega$. It is known that $Z(\Omega)$ is a complete metric space. We denote by $M(\Omega)$ the closure in $Z(\Omega)$ of the set of all rational functions. Obviously, $M(\Omega)$ is a complete metric space itself, as a closed subspace of $Z(\Omega)$. See also \cite{NESTORIDIS3}. In a similar way to Theorems \ref{theorem 4.1} and \ref{theorem 4.4}, we obtain the following results.

\begin{theorem} \label{theorem 4.7}

Let $\Omega \subseteq \mathbb{C}$ be an open set and $\zeta \in \Omega$ be a fixed element. Let also $L \subseteq \Omega$ and $K \subseteq \mathbb{C} \setminus \Omega$ be two compact sets. We consider a sequence $(p_n)_{n \geq 1} \subseteq \mathbb{N}$ with $p_n \to + \infty$. Now, for every $n \in \mathbb{N}$ let $q_{1}^{(n)}, q_{2}^{(n)}, \cdots, q_{N(n)}^{(n)} \in \mathbb{N}$, where $N(n)$ is another natural number. Suppose also that $\min \{ q_{1}^{(n)}, q_{2}^{(n)}, \cdots, q_{N(n)}^{(n)} \} \to + \infty$. Then there exists a function $f \in M(\Omega)$ with $f(\zeta) \neq \infty$ such that for every rational function $h$ there exists a subsequence $(p_{k_n})_{n \geq 1}$ of the sequence $(p_{n})_{n \geq 1}$ satisfying the following:

\begin{itemize}

\item[(1)]
$f \in D_{p_{k_n}, q_{j}^{(k_n)}}(\zeta)$ for every $n \in \mathbb{N}$ and for every $j \in \{ 1, \cdots, N(k_n) \}$.

\item[(2)]
$\max_{j = 1, \cdots, N(k_n)} \sup_{z \in K} \chi([f; p_{k_n} / q_{j}^{(k_n)}]_{\zeta}(z), h(z)) \to 0 \; \text{as} \; n \to + \infty$.

\item[(3)]
$\max_{j = 1, \cdots, N(k_n)} \sup_{z \in L} \chi([f; p_{k_n} / q_{j}^{(k_n)}]_{\zeta}(z), f(z)) \to 0 \; \text{as} \; n \to + \infty$.

\end{itemize}

Moreover, the set of all functions $f$ satisfing $(1) - (3)$ is dense and $G_{\delta}$ in $M(\Omega)$.

\end{theorem}

\begin{proof} We apply Theorem \ref{theorem 4.1} for $L = L_k$ and for $K = K_m$ and in that way we obtain a $G_{\delta}$ - dense class of $M(\Omega)$; the class $\mathcal{U}(k, m)$. If $\mathcal{U}$ is the class of all functions of $M(\Omega)$ satisfying the result, then it holds:
$$ \mathcal{U} = \bigcup_{n, m \in \mathbb{N}} \mathcal{U}(k, m). $$

The reult follows from Baire's theorem.

\end{proof}

Next we state some other results whose proofs are similar to the previous ones and are omitted.

\begin{theorem} \label{theorem 4.8}

Let $\Omega \subseteq \mathbb{C}$ be an open set and $\zeta \in \Omega$ be a fixed point. We consider a sequence $(p_n)_{n \geq 1} \subseteq \mathbb{N}$ with $p_n \to + \infty$. Now, for every $n \in \mathbb{N}$ let $q_{1}^{(n)}, q_{2}^{(n)}, \cdots, q_{N(n)}^{(n)} \in \mathbb{N}$, where $N(n)$ is another natural number. Suppose also that $\min \{ q_{1}^{(n)}, q_{2}^{(n)}, \cdots, q_{N(n)}^{(n)} \} \to + \infty$. Then there exists a function $f \in M(\Omega)$ with $f(\zeta) \neq \infty$ satisfying the following:

For every compact set $K \subseteq \mathbb{C} \setminus \Omega$, for every compact set $L \subseteq \Omega$ and for every rational function $h$, there exists a subsequence $(p_{k_n})_{n \geq 1}$ of the sequence $(p_n)_{n \geq 1}$ such that:

\begin{itemize}

\item[(1)]
$f \in D_{p_{k_n}, q_{j}^{(k_n)}}(\zeta)$ for every $n \in \mathbb{N}$ and for every $j \in \{ 1, \cdots, N(k_n) \}$.

\item[(2)]
$\max_{j = 1, \cdots, N(k_n)} \sup_{z \in K} \chi([f; p_{k_n} / q_{j}^{(k_n)}]_{\zeta}(z), h(z)) \to 0 \; \text{as} \; n \to + \infty$.

\item[(3)]
$\max_{j = 1, \cdots, N(k_n)} \sup_{z \in L} \chi([f; p_{k_n} / q_{j}^{(k_n)}]_{\zeta}(z), f(z)) \to 0 \; \text{as} \; n \to + \infty$.

\end{itemize}

Moreover, the set of all functions $f$ satisfing $(1) - (3)$ is dense and $G_{\delta}$ in $M(\Omega)$.

\end{theorem}

We interchange the roles of $p$ and $q$.

\begin{theorem} \label{theorem 4.9}

Let $\Omega \subseteq \mathbb{C}$ be an open set and $\zeta \in \Omega$ be a fixed element. We consider a sequence $(q_n)_{n \geq 1} \subseteq \mathbb{N}$ with $q_n \to + \infty$. Now, for every $n \in \mathbb{N}$ let $p_{1}^{(n)}, p_{2}^{(n)}, \cdots, p_{N(n)}^{(n)} \in \mathbb{N}$, where $N(n)$ is another natural number. Suppose also that $\min \{ p_{1}^{(n)}, p_{2}^{(n)}, \cdots, p_{N(n)}^{(n)} \} \to + \infty$. Then there exists a function $f \in M(\Omega)$ with $f(\zeta) \neq \infty$ satisfying the following:

For every compact set $K \subseteq \mathbb{C} \setminus \Omega$, for every compact set $L \subseteq \Omega$ and for every rational function $h$, there exists a subsequence $(q_{k_n})_{n \geq 1}$ of the sequence $(q_{n})_{n \geq 1}$ such that:

\begin{itemize}

\item[(1)]
$f \in D_{p_{j}^{(k_n)}, q_{k_n}}(\zeta)$ for every $n \in \mathbb{N}$ and for every $j \in \{ 1, \cdots, N(k_n) \}$.

\item[(2)]
$\max_{j = 1, \cdots, N(k_n)} \sup_{z \in K} \chi([f; p_{j}^{(k_n)} / q_{k_n}]_{\zeta}(z), h(z)) \to 0 \; \text{as} \; n \to + \infty$.

\item[(3)]
$\max_{j = 1, \cdots, N(k_n)} \sup_{z \in L} \chi([f; p_{j}^{(k_n)} / q_{k_n}]_{\zeta}(z), f(z)) \to 0 \; \text{as} \; n \to + \infty$.

\end{itemize}

Moreover, the set of all functions $f$ satisfing $(1) - (3)$ is dense and $G_{\delta}$ in $M(\Omega)$.

\end{theorem}

We present now two generic results for Universal Pad\'{e} approximants of Type II for formal power series, without proofs which are similar to the previous ones. Thus we strengthen results from \cite{DARAS.NESTORIDIS.PAPADIMITROPOULOS}. 

\begin{theorem} \label{theorem 4.10}

Let $(p_n)_{n \geq 0} \subseteq \mathbb{N}$ be a sequence such that $p_n \to + \infty$. Now, for every $n \in \mathbb{N}$, let $q^{(n)}_{0}, \cdots, q^{(n)}_{M(n)} \in \mathbb{N}$, where $M(n) \in \mathbb{N}$. Suppose also that $\min \{ q^{(n)}_{0}, \cdots, q^{(n)}_{M(n)} \} \to + \infty$. Then there exists a formal power series $f(z) = \sum_{n = 0}^{+ \infty} a_n z^n$ with $(a_n)_{n \geq 0} \in \mathbb{C}^{\mathbb{N}}$, satisfying the following:

For every compact set $K \subseteq \mathbb{C} \setminus \{ 0 \}$ and for every function $h(z) \in M(\mathbb{C} \setminus \{ 0 \})$, there exists a subsequence $(p_{k_n})_{n \geq 0}$ of the sequence $(p_n)_{n \geq 0}$ such that:

\begin{itemize}

\item[(1)]
$f \in D_{p_{k_n}, q^{(k_n)}_{j}}$ for every $j \in \{ 0, \cdots, M(k_n) \}$ and for every $n \in \mathbb{N}$.

\item[(2)]
$\max_{j \in \{ 0, \cdots, M(k_n) \}} \sup_{z \in K}{\chi([f; p_{k_n} / q^{(k_n)}_{j}](z), h(z))} \to 0$, as $n \to + \infty$.

\end{itemize}

Moreover, the set of all functions satisfying $(1)$ \& $(2)$ is $G_{\delta}$ - dense in both spaces $(\mathbb{C}^{\mathbb{N}}, \rho_c)$ and  $(\mathbb{C}^{\mathbb{N}}, \rho_d)$.

\end{theorem}

\begin{theorem} \label{theorem 4.11}

Let $(q_n)_{n \geq 0} \subseteq \mathbb{N}$ be a sequence such that $q_n \to + \infty$. Now, for every $n \in \mathbb{N}$, let $p^{(n)}_{0}, \cdots, p^{(n)}_{M(n)} \in \mathbb{N}$, where $M(n) \in \mathbb{N}$. Suppose also that $\min \{ p^{(n)}_{0}, \cdots, p^{(n)}_{M(n)} \} \to + \infty$. Then there exists a formal power series $f(z) = \sum_{n = 0}^{+ \infty} a_n z^n$ with $(a_n)_{n \geq 0} \in \mathbb{C}^{\mathbb{N}}$, satisfying the following:

For every compact set $K \subseteq \mathbb{C} \setminus \{ 0 \}$ and for every function $h(z) \in M(\mathbb{C} \setminus \{ 0 \})$, there exists a subsequence $(q_{k_n})_{n \geq 0}$ of the sequence $(q_n)_{n \geq 0}$ such that:

\begin{itemize}

\item[(1)]
$f \in D_{p^{(k_n)}_{j}, q_{k_n}}$ for every $j \in \{ 0, \cdots, M(k_n) \}$ and for every $n \in \mathbb{N}$.

\item[(2)]
$\max_{j \in \{ 0, \cdots, M(k_n) \}} \sup_{z \in K}{\chi([f; p^{(k_n)}_{j} / q_{k_n}](z), h(z))} \to 0$, as $n \to + \infty$.

\end{itemize}

Moreover, the set of all functions satisfying $(1)$ \& $(2)$ is $G_{\delta}$ - dense in in both spaces $(\mathbb{C}^{\mathbb{N}}, \rho_c)$ and  $(\mathbb{C}^{\mathbb{N}}, \rho_d)$.

\end{theorem}

Now we present some generic results on $X^{\infty}(\Omega)$ strenghtening results from \cite{NESTORIDIS.ZADIK}.

Let $\Omega \subseteq \mathbb{C}$ be an open set. We say that a function $f \in H(\Omega)$ belongs to $A^{\infty}(\Omega)$ if for every $n \in \mathbb{N}$ the $n$th derivative $f^{(n)}$ of $f$ can be extended continuously on $\overline{\Omega}$.

In $A^{\infty}(\Omega)$ we consider the topology defined by the family of seminorms $\sup_{z \in L_k} |f^{(n)}(z)|$, where $\{ L_k \}_{k \geq 1}$ is a family of compact subsets of $\overline{\Omega}$, such that for every compact subset $L \subseteq \Omega$, there exists an index $n \equiv n(L) \in \mathbb{N}$ satisfying $L \subseteq L_n$. For instance, it suffices to consider the family $L_n = \overline{\Omega} \cap \overline{B}(0, n)$ for every $n \in \mathbb{N}$. It is known that with this topology $A^{\infty}(\Omega)$ becomes a Fr\'{e}chet space.

The space $X^{\infty}(\Omega)$ is the closure in $A^{\infty}(\Omega)$ of all rational functions with poles off $\overline{\Omega}$. It is obvious that $X^{\infty}(\Omega)$ is a complete metric space itself, as a closed subset of $A^{\infty}(\Omega)$.

\begin{theorem} \label{theorem 4.12}

Let $\Omega \subseteq \mathbb{C}$ be an open set and $L, L' \subseteq \overline{\Omega}$, $K \subseteq \mathbb{C} \setminus \overline{\Omega}$ be some compact sets. Let also $(p_n)_{n \geq 1} \subseteq \mathbb{N}$ be a sequence such that $p_n \to + \infty$. Now for every $n \geq 1$, let $q^{(n)}_{1}, \cdots, q^{(n)}_{M(n)} \in \mathbb{N}$, where $M(n) \in \mathbb{N}$ and suppose that $\min \{ q^{(n)}_{1}, \cdots, q^{(n)}_{M(n)} \} \to + \infty$.

Then there exists a function $f \in X^{\infty}(\Omega)$ such that for every rational function $h$, there exists a subsequence $(p_{k_n})_{n \geq 1}$ of the sequence $(p_n)_{n \geq 1}$ satisfying the following:  

\begin{itemize}

\item[(1)]
$f \in D_{p_{k_n}, q^{(k_n)}_{j}}(\zeta)$ for every $\zeta \in L$, for every $j \in \{ 1, \cdots, M(k_n) \}$ and for every $n \in \mathbb{N}$.

\item[(2)]
For every $l \in \mathbb{N}$, $\max_{j \in \{ 1, \cdots, M(k_n) \}} \sup_{\zeta \in L} \sup_{z \in L'} |[f; p_{k_n} / q^{(k_n)}_{j}]^{(l)}_{\zeta}(z) - f^{(l)}(z)| \to 0$ as $n \to + \infty$.

\item[(3)]
$\max_{j \in \{ 1, \cdots, M(k_n) \}} \sup_{\zeta \in L} \sup_{z \in K} \chi([f; p_{k_n} / q^{(k_n)}_{j}]_{\zeta}(z), h(z)) \to 0$ as $n \to + \infty$.

\end{itemize}

Moreover, the set of all functions $f \in X^{\infty}(\Omega)$ satisfying $(1) - (3)$ is dense and $G_{\delta}$ in $X^{\infty}(\Omega)$.

\end{theorem}

\begin{theorem} \label{theorem 4.13}

Let $\Omega \subseteq \mathbb{C}$ be an open set and $\zeta \in \overline{\Omega}$ be a fixed point. Let also $(p_n)_{n \geq 1} \subseteq \mathbb{N}$ be a sequence such that $p_n \to + \infty$. Now for every $n \geq 1$, let $q^{(n)}_{1}, \cdots, q^{(n)}_{M(n)} \in \mathbb{N}$, where $M(n) \in \mathbb{N}$ and suppose that $\min \{ q^{(n)}_{1}, \cdots, q^{(n)}_{M(n)} \} \to + \infty$. Then there exists a function $f \in X^{\infty}(\Omega)$ such that for every rational function $h$ and for every compact set $K \subseteq \mathbb{C} \setminus \overline{\Omega}$, there exists a subsequence $(p_{k_n})_{n \geq 1}$ of the sequence $(p_n)_{n \geq 1}$ satisfying the following:

\begin{itemize}

\item[(1)]
$f \in D_{p_{k_n}, q^{(k_n)}_{j}}(\zeta)$ for every $j \in \{ 1, \cdots, M(k_n) \}$ and for every $n \in \mathbb{N}$.

\item[(2)]
For every $l \in \mathbb{N}$ and for every compact set $L \subseteq \overline{\Omega}$ it holds:
$$ \max_{j \in \{ 1, \cdots, M(k_n) \}} \sup_{z \in L} |[f; p_{k_n} / q^{(k_n)}_{j}]^{(l)}_{\zeta}(z) - f^{(l)}(z)| \to 0. $$ 

as $n \to + \infty$

\item[(3)]
$\max_{j \in \{ 1, \cdots, M(k_n) \}} \sup_{z \in K} \chi([f; p_{k_n} / q^{(k_n)}_{j}]_{\zeta}(z), h(z)) \to 0$ as $n \to + \infty$.

\end{itemize}

Moreover, the set of all functions $f \in X^{\infty}(\Omega)$ satisfying $(1) - (3)$ is dense and $G_{\delta}$ in $X^{\infty}(\Omega)$.

\end{theorem}

\begin{theorem} \label{theorem 4.14}

Let $\Omega \subseteq \mathbb{C}$ be an open set and $(p_n)_{n \geq 1} \subseteq \mathbb{N}$ be a sequence such that $p_n \to + \infty$. Now for every $n \geq 1$, let $q^{(n)}_{1}, \cdots, q^{(n)}_{M(n)} \in \mathbb{N}$, where $M(n) \in \mathbb{N}$ and suppose that $\min \{ q^{(n)}_{1}, \cdots, q^{(n)}_{M(n)} \} \to + \infty$. Then there exists a function $f \in X^{\infty}(\Omega)$ satisfying the following:

For every compact set $K \subseteq \mathbb{C} \setminus \overline{\Omega}$ and for every rational function $h$, there exists a subsequence $(p_{k_n})_{n \geq 1}$ of the sequence $(p_n)_{n \geq 1}$ satisfying the following:  

\begin{itemize}

\item[(1)]
For every compact set $L \subseteq \overline{\Omega}$ there exists an index $n(L) \in \mathbb{N}$ such that $f \in D_{p_{k_n}, q^{(k_n)}_{j}}(\zeta)$ for every $\zeta \in L$, for every $j \in \{ 1, \cdots, M(k_n) \}$ and for every $n \geq n(L)$.

\item[(2)]
For every $l \in \mathbb{N}$ and for every compact set $L \subseteq \overline{\Omega}$ it holds:
$$ \max_{j \in \{ 1, \cdots, M(k_n) \}} \sup_{\zeta \in L} \sup_{z \in L} |[f; p_{k_n} / q^{(k_n)}_{j}]^{(l)}_{\zeta}(z) - f^{(l)}(z)| \to 0 $$ 

as $n \to + \infty$.

\item[(3)]
For every compact set $L \subseteq \overline{\Omega}$, it holds:
$$ \max_{j \in \{ 1, \cdots, M(k_n) \}} \sup_{\zeta \in L} \sup_{z \in K} \chi([f; p_{k_n} / q^{(k_n)}_{j}]_{\zeta}(z), h(z)) \to 0 $$ 

as $n \to + \infty$.

\end{itemize}

Moreover, the set of all functions $f \in X^{\infty}(\Omega)$ satisfying $(1) - (3)$ is dense and $G_{\delta}$ in $X^{\infty}(\Omega)$.

\end{theorem}

Now we interchange the roles of $p$ and $q$.

\begin{theorem} \label{theorem 4.15}

Let $\Omega \subseteq \mathbb{C}$ be an open set and $L, L' \subseteq \overline{\Omega}$, $K \subseteq \mathbb{C} \setminus \overline{\Omega}$ be some compact sets. Let also $(q_n)_{n \geq 1} \subseteq \mathbb{N}$ be a sequence such that $q_n \to + \infty$. Now for every $n \geq 1$, let $p^{(n)}_{1}, \cdots, p^{(n)}_{M(n)} \in \mathbb{N}$, where $M(n) \in \mathbb{N}$ and suppose that $\min \{ p^{(n)}_{1}, \cdots, p^{(n)}_{M(n)} \} \to + \infty$.

Then there exists a function $f \in X^{\infty}(\Omega)$ such that for every rational function $h$, there exists a subsequence $(q_{k_n})_{n \geq 1}$ of the sequence $(q_n)_{n \geq 1}$ satisfying the following:  

\begin{itemize}

\item[(1)]
$f \in D_{p^{(k_n)}_{j}, q_{k_n}}(\zeta)$ for every $\zeta \in L$, for every $j \in \{ 1, \cdots, M(k_n) \}$ and for every $n \in \mathbb{N}$.

\item[(2)]
For every $l \in \mathbb{N}$, $\max_{j \in \{ 1, \cdots, M(k_n) \}} \sup_{\zeta \in L} \sup_{z \in L'} |[f; p^{(k_n)}_{j} / q_{k_n}]^{(l)}_{\zeta}(z) - f^{(l)}(z)| \to 0$ as $n \to + \infty$.

\item[(3)]
$\max_{j \in \{ 1, \cdots, M(k_n) \}} \sup_{\zeta \in L} \sup_{z \in K} \chi([f; p^{(k_n)}_{j} / q_{k_n}]_{\zeta}(z), h(z)) \to 0$ as $n \to + \infty$.

\end{itemize}

Moreover, the set of all functions $f \in X^{\infty}(\Omega)$ satisfying $(1) - (3)$ is dense and $G_{\delta}$ in $X^{\infty}(\Omega)$.

\end{theorem}

\begin{theorem} \label{theorem 4.16}

Let $\Omega \subseteq \mathbb{C}$ be an open set and $\zeta \in \overline{\Omega}$ be a fixed point. Let also $(q_n)_{n \geq 1} \subseteq \mathbb{N}$ be a sequence such that $q_n \to + \infty$. Now for every $n \geq 1$, let $p^{(n)}_{1}, \cdots, p^{(n)}_{M(n)} \in \mathbb{N}$, where $M(n) \in \mathbb{N}$ and suppose that $\min \{ p^{(n)}_{1}, \cdots, p^{(n)}_{M(n)} \} \to + \infty$. Then there exists a function $f \in X^{\infty}(\Omega)$ such that for every compact set $K \subseteq \mathbb{C} \setminus \overline{\Omega}$ and for every rational function $h$, there exists a subsequence $(q_{k_n})_{n \geq 1}$ of the sequence $(q_n)_{n \geq 1}$ satisfying the following:  

\begin{itemize}

\item[(1)]
$f \in D_{p^{(k_n)}_{j}, q_{k_n}}(\zeta)$ for every $j \in \{ 1, \cdots, M(k_n) \}$ and for every $n \in \mathbb{N}$.

\item[(2)]
For every $l \in \mathbb{N}$ and for every compact set $L \subseteq \overline{\Omega}$ it holds:
$$ \max_{j \in \{ 1, \cdots, M(k_n) \}} \sup_{z \in L} |[f; p^{(k_n)}_{j} / q_{k_n}]^{(l)}_{\zeta}(z) - f^{(l)}(z)| \to 0 $$ 

as $n \to + \infty$.

\item[(3)]
$\max_{j \in \{ 1, \cdots, M(k_n) \}} \sup_{z \in K} \chi([f; p^{(k_n)}_{j} / q_{k_n}]_{\zeta}(z), h(z)) \to 0$ as $n \to + \infty$.

\end{itemize}

Moreover, the set of all functions $f \in X^{\infty}(\Omega)$ satisfying $(1) - (3)$ is dense and $G_{\delta}$ in $X^{\infty}(\Omega)$.

\end{theorem}

\begin{theorem} \label{theorem 4.17}

Let $\Omega \subseteq \mathbb{C}$ be an open set and $(q_n)_{n \geq 1} \subseteq \mathbb{N}$ be a sequence such that $q_n \to + \infty$. Now for every $n \geq 1$, let $p^{(n)}_{1}, \cdots, p^{(n)}_{M(n)} \in \mathbb{N}$, where $M(n) \in \mathbb{N}$ and suppose that $\min \{ p^{(n)}_{1}, \cdots, p^{(n)}_{M(n)} \} \to + \infty$. Then there exists a function $f \in X^{\infty}(\Omega)$ satisfying the following:

For every compact set $K \subseteq \mathbb{C} \setminus \overline{\Omega}$ and for every rational function $h$, there exists a subsequence $(q_{k_n})_{n \geq 1}$ of the sequence $(q_n)_{n \geq 1}$ satisfying the following:  

\begin{itemize}

\item[(1)]
For every compact set $L \subseteq \overline{\Omega}$ there exists an index $n(L) \in \mathbb{N}$ such that $f \in D_{p^{(k_n)}_{j}, q_{k_n}}(\zeta)$ for every $\zeta \in L$, for every $j \in \{ 1, \cdots, M(k_n) \}$ and for every $n \geq n(L)$.

\item[(2)]
For every $l \in \mathbb{N}$ and for every compact set $L \subseteq \overline{\Omega}$ it holds:
$$ \max_{j \in \{ 1, \cdots, M(k_n) \}} \sup_{\zeta \in L} \sup_{z \in L} |[f; p^{(k_n)}_{j} / q_{k_n}]^{(l)}_{\zeta}(z) - f^{(l)}(z)| \to 0 $$ 

as $n \to + \infty$.

\item[(3)]
For every compact set $L \subseteq \overline{\Omega}$, it holds:
$$ \max_{j \in \{ 1, \cdots, M(k_n) \}} \sup_{\zeta \in L} \sup_{z \in K} \chi([f; p^{(k_n)}_{j} / q_{k_n}]_{\zeta}(z), h(z)) \to 0 $$ 

as $n \to + \infty$.

\end{itemize}

Moreover, the set of all functions $f \in X^{\infty}(\Omega)$ satisfying $(1) - (3)$ is dense and $G_{\delta}$ in $X^{\infty}(\Omega)$.

\end{theorem}

\section{\sloppy{Affine and algebraic genericities of two classes of functions}} 

In this section we deal with the class of functions on a simply connected domain $\Omega \subseteq \mathbb{C}$ which satisfy the requirements of Theorem \ref{theorem 1.1} in the Introduction for a fixed center of expansion $\zeta \in \Omega$: the class $\mathcal{A}$.

We construct a particular function $f$ in the above class. Our construction is based on the observation that $[f; p / q]_{\zeta}(z) \equiv S_{p}(f, \zeta)(z)$ with $q \geq 1$ if and only if $a_{p + 1} = a_{p + 2} = \cdots = a_{p + q} = 0$, where $f(z) = \sum_{n = 0}^{+ \infty} a_n (z - \zeta)^n$ is the Taylor expansion of the function $f$, centered at $\zeta \in \Omega$.

\begin{theorem} \label{theorem 5.1} 
Let $\Omega \subseteq \mathbb{C}$ be a simply connected domain and $\zeta \in \Omega$ be a fixed point. Let also $(p_n)_{n \geq 1} \subseteq \mathbb{N}$ be a sequence such that $p_n \to + \infty$. Also, for every $n \in \mathbb{N}$, let $N(n) \in \mathbb{N}$ and $q_{1}^{(n)}, \cdots, q_{N(n)}^{(n)} \in \mathbb{N}$.

Then there exists a function $f \in H(\Omega)$, $f(z) = \sum_{n = 0}^{+ \infty} a_n (z - \zeta)^n$ satisfying the following: 

For every compact set $K \subseteq \mathbb{C} \setminus \Omega$ with connected complement and for every function $h \in A(K)$, there exists a subsequence $(p_{k_n})_{n \geq 1}$ of the sequence $(p_n)_{n \geq 1}$ such that:

\begin{itemize}

\item[(1)]
$\sup_{z \in K} |S_{p_{k_n}}(f, \zeta)(z) - h(z)| \to 0$, as $n \to + \infty$.

\item[(2)]
$\sup_{z \in J} |S_{p_{k_n}}(f, \zeta)(z) - f(z)| \to 0$, as $n \to + \infty$ for every compact set $J \subseteq \Omega$.

\end{itemize}

Furthermore, for every $n \in \mathbb{N}$ it holds $a_{p_{k_n}} \neq 0$  and $a_{p_{k_n} + s} = 0$ for every $s = 1, \cdots, \max\{ q_{1}^{(k_n)}, \cdots, q_{N(k_n)}^{(k_n)}\}$.

\end{theorem}

\begin{proof} Let $\{ f_j \}_{j \geq 1}$ be an enumeration of polynomials with coefficients of $\mathbb{Q} + i\mathbb{Q}$. Let also $\{ K_m \}_{m \geq 1}$ and $\{ L_k \}_{k \geq 1}$ be two families of compact subsetes of $\mathbb{C}$ satisfying Lemmas \ref{lemma 3.1} and \ref{lemma 3.2} respectively. The set $\{ (K_m, f_j) : m, j \geq 1 \}$ is infinite denumerable and thus we consider a function $t: \mathbb{N} \to \mathbb{N}$ such that $\{ (K_m, f_j) : m, j \geq 1 \} = \{ (K_{m_t}, f_{j_t}) : t \geq 1 \}$, where we suppose that each pair $(K_{m_t}, f_{j_t})$ appears infinitely many times. In any other case, see Remark \ref{remark 5.2}.

\begin{itemize}

\item[\textbf{Step 1}.]

We consider the function $w_1(z) = 
\begin{cases} f_{j_1}(z) &\mbox{if } \; z \in K_{m_1} \\
0 &\mbox{if } \; z \in L_{1} \end{cases}$.

We notice that $w_1 \in A(K_{m_1} \cup L_{1})$. We apply Mergelyan's theorem and we find a polynomial $h_1$ such that $\sup_{z \in K_{m_1} \cup L_{1}} |w_1(z) - h_1(z)| < 1$. We consider an index $k_1 \in \mathbb{N}$ such that $deg(h_1(z)) < p_{k_1}$. Next, we select a $c_1 \in \mathbb{C} \setminus \{ 0 \}$ such that:
$$ \sup_{z \in K_{m_1} \cup L_{1}} |w_1(z) - (h_1(z) + c_1(z - \zeta)^{p_{k_1}})| < 1. $$

Note that such a choice is possible. We set $H_1(z) = h_1(z) + c_1(z - \zeta)^{p_{k_1}}$. Clearly the function $H_1$ is a polynomial with $deg(H_1(z)) = p_{k_1}$. Finally, we select $t_1 \geq 1 + \max\{ q_{1}^{(1)}, \cdots, q_{N(1)}^{(1)}\}$.

\item[\textbf{Step 2}.]

We consider the function $w_2(z) = 
\begin{cases} \frac{f_{j_2}(z) - H_1(z)}{(z - \zeta)^{p_{k_1} + t_1}} &\mbox{if } \; z \in K_{m_2} \\
0 &\mbox{if } \; z \in L_{2} \end{cases}$

We notice that $w_1 \in A(K_{m_2} \cup L_{2})$. We apply Mergelyan's theorem and we find a polynomial $h_2$ such that:
$$ \sup_{z \in K_{m_2} \cup L_{2}} |w_2(z) - h_2(z)| < \frac{1}{2^2} \cdot \frac{1}{\max_{z \in K_{m_2} \cup L_{2}} |z - \zeta|^{p_{k_1} + t_1} + 1}. $$ 

We consider an index $k_2 \in \mathbb{N}$ such that $deg(h_2(z) \cdot (z - \zeta)^{p_{k_1} + t_1}) < p_{k_2}$. Next, we select a $c_2 \in \mathbb{C} \setminus \{ 0 \}$ such that:
$$ \sup_{z \in K_{m_2} \cup L_{2}} |w_2(z) - ((z - \zeta)^{p_{k_1} + t_1} h_2(z) + c_2(z - \zeta)^{p_{k_2}}| < \frac{1}{2^2}. $$

Note that such a choice is possible. We set $H_2(z) = (z - \zeta)^{p_{k_1} + t_1} h_2(z) + c_2(z - \zeta)^{p_{k_2}}$. Clearly the function $H_2$ is a polynomial with $deg(H_2(z)) = p_{k_2}$. Finally, we select $t_2 \geq 1 + \max\{ q_{1}^{(2)}, \cdots, q_{N(2)}^{(2)}\}$.

\item[\textbf{Step n}.] So far we have defined the polynomials $H_1, \cdots, H_{n - 1}$ with $deg(H_i(z)) = p_{k_i}$ for every $i = 1, \cdots, n - 1$. We consider the function: 
$$ w_n(z) = 
\begin{cases} \frac{f_{j_n}(z) - (H_{1}(z) + \sum_{N = 1}^{n - 2} (z - \zeta)^{p_{k_N} + t_N} H_{N + 1})}{(z - \zeta)^{p_{k_{n - 1}} + t_{n - 1}}} &\mbox{if } \; z \in K_{m_n} \\
0 &\mbox{if } \; z \in L_{n} \end{cases}. $$

We notice that $w_n \in A(K_{m_n} \cup L_{n})$. We apply Mergelyan's theorem and we find a polynomial $h_n$ such that:
$$ \sup_{z \in K_{m_{n}} \cup L_{n}} |w_n(z) - h_n(z)| < \frac{1}{n^2} \cdot \frac{1}{\max_{z \in K_{m_n} \cup L_{n}} |z - \zeta|^{p_{k_{n - 1}} + t_{n - 1}} + 1}. $$ 

We consider an index $k_n \in \mathbb{N}$ such that:
$$ deg((z - \zeta)^{p_{k_{n - 1}} + t_{n - 1}} \cdot h_n(z)) < p_{k_n}. $$

Next, we select a $c_n \in \mathbb{C} \setminus \{ 0 \}$ such that:
$$ \sup_{z \in K_{m_n} \cup L_{n}} |w_n(z) - ((z - \zeta)^{p_{k_{n - 1}}} h_n(z) + c_n(z - \zeta)^{p_{k_n}})| < \frac{1}{n^2}. $$

Note that such a choice is possible. We set $H_n(z) = (z - \zeta)^{p_{k_{n - 1}}} h_n(z) + c_n(z - \zeta)^{p_{k_n}}$. Clearly the function $H_n$ is a polynomial with $deg(H_n(z)) = p_{k_n}$. 

\end{itemize}

From Weierstrass's theorem, the sequence of polynomials:
$$ \{ H_{1}(z) + H_{2}(z) + \cdots + H_{n}(z) + \cdots \}_{n \geq 2} $$ 

converges uniformly on every compact subset $J \subseteq \Omega$ to a function $f \in H(\Omega)$. We notice that from Cauchy's integral formula, for every $n \in \mathbb{N}$ it holds:
$$ S_{p_{k_n}}(f, \zeta)(z) = H_{1}(z) + \sum_{N = 1}^{p_{k_n} - 1} H_{N + 1}(z). $$

We will show that the function $f$ meets the requirements of Theorem \ref{theorem 5.1}. 

\begin{itemize}

\item[$\bullet$]

Let $J \subseteq \Omega$ be a compact set. We consider an index $k \in \mathbb{N}$ such that $J \subseteq L_k$. Then:
$$ \sup_{z \in L_k} |S_{p_{k_n}}(f, \zeta)(z) - f(z)| \leq \sum_{s = M(n)}^{+ \infty} \frac{1}{s^2} \to 0 $$

as $n \to + \infty$, where $M(n) \in \mathbb{N}$ and $M(n) \to +\infty$ as $n \to + \infty$ as well.

\item[$\bullet$]

Let $K \subseteq \mathbb{C} \setminus \Omega$ be a compact set with connected complement, $h \in A(K)$ and $\varepsilon > 0$. We consider an index $m \in \mathbb{N}$ such that $K \subseteq K_m$. We use Mergelyan's theorem in order to find a polynomial $f_j$ such that $|| h - f_j ||_{K} < \frac{\varepsilon}{3}$. Then, according to our initial hypothesis, $(K_m, f_j) = (K_{m_t}, f_{j_t})$ for infinitely many $t \in \mathbb{N}$. Also, according to the construction of $f$, it holds $|| f_{j_t} - S_{p_{k_{j_t}}}(f, \zeta) ||_{K_{m_t}} < \frac{1}{t^2}$ for infinitely many $t \in \mathbb{N}$. Equivalently, $|| f_{j} - S_{p_{k_{j_t}}}(f, \zeta) ||_{K_{m}} < \frac{1}{t^2}$ for infinitely many $t \in \mathbb{N}$. For $t \to + \infty$, we select a $t_0 \in \mathbb{N}$ such that $|| f_{j} - S_{p_{k_{j_{t_0}}}}(f, \zeta) ||_{K_{m}} < \frac{1}{{t_0}^2} < \frac{\varepsilon}{3}$. The triangle inequality yields now the result.

\end{itemize}

\end{proof}

\begin{remark} \label{remark 5.2} We notice that in the previous proof we can also use an enumeration $(K_{m_t}, f_{j_t}), t \in \mathbb{N}$ where every pair $(K_m, f_j)$ appears only once. This alters only the last part of the proof, where we consider a compact set $K \subseteq \mathbb{C} \setminus \Omega$ with connected complement, a function $h \in A(K)$ and we want to approximate uniformly on $K$ the function $h$ by a subsequence of the partial sums $\{ S_{n}(f, \zeta) \}_{n \geq 1}$ of the function $f$ we have already constructed.

Let $\varepsilon > 0$. We consider an index $m \in \mathbb{N}$ such that $K \subseteq K_m$. We use Mergelyan's theorem in order to find a polynomial $f_j$ such that $|| h - f_j ||_{K} < \frac{\varepsilon}{3}$. Suppose that $(K_{m_t}, f_{j_t}) = (K_m, f_j)$ only once. According to the construction of the function $f$, it holds $|| f_{j_t} - S_{p_{k_{j_t}}}(f, \zeta) ||_{K_{m_t}} < \frac{1}{t^2}$ for each such $t \in \mathbb{N}$. We notice that for every $q \in (0, \frac{\varepsilon}{3}) \cap \mathbb{Q}$ the function $f_j + q$ is a polynomial with coefficients in $\mathbb{Q} + i \mathbb{Q}$, thus there exists an index $j(q) \in \mathbb{N}$ such that $f_j + q \equiv f_{j(q)}$. Also, for $q, q' \in (0, \frac{\varepsilon}{3}) \cap \mathbb{Q}$ with $q \neq q'$ it holds $j(q) \neq j(q')$. In other words, there exist infinitely many indexes $i \in \mathbb{N}$ such that $|| f_j - f_i ||_{K_m} < \frac{\varepsilon}{3}$. That allows us to find an index  $j_0 \in \mathbb{N}$ large enough such that $|| f_j - f_{j_0} ||_{K_m} < \frac{\varepsilon}{3}$ and also $\frac{1}{j^{2}_t} < \frac{\varepsilon}{3}$. Thus $|| S_{p_{k_{j_t}}}(f, \zeta) - f_{j_t} ||_{K_m} < \frac{1}{j^{2}_t} < \frac{\varepsilon}{3}$. The result follows now from the triangle inequality. See also \cite{MELAS.NESTORIDIS.PAPADOPERAKIS}.

We mention that the above construction is a modification of the construction of universal Taylor series without use of Baire's theorem (\cite{CHUI.PARNES} \cite{LUH} \cite{NESTORIDIS1} \cite{MOUZE.NESTORIDIS.PAPADOPERAKIS.TSIRIVAS}).

\end{remark}

\begin{definition} \label{definition 5.3}

We denote as $\mathcal{B} \equiv \mathcal{B}(\mathcal{A}^{(1)})$ the set of all functions satisfying Theorem \ref{theorem 5.1}. 

\end{definition}

\begin{proposition} \label{proposition 5.4}

Let $f \in \mathcal{B}(\mathcal{A}^{(1)})$ and $P$ be a polynomial. Then $f + P \in \mathcal{B}(\mathcal{A}^{(1)}) \subseteq \mathcal{A}$. Thus, the class $\mathcal{A}$ contains an affine dense subspace of $H(\Omega)$. It follows that $\mathcal{A}$ is affinely generic.

\end{proposition}

\begin{proof} \sloppy{This is obvious, according to the construction of the class} $\mathcal{B}(\mathcal{A}^{(1)}) \subseteq \mathcal{A}$, since the condition $a_{p_{k_n} + s} = 0$ for every $s = 1, \cdots, \max\{ q_{1}^{(k_n)}, \cdots, q_{N(k_n)}^{(k_n)}\}$ implies that $[f; p_{k_n} / q_{j}^{(k_n)}]_{\zeta} \equiv S_{p_{k_n}}(f, \zeta)$ for every $j = 1, \cdots, M(k_n)$. Further, the above equations combined with $a_{p_{k_n}} \neq 0$ imply that $f \in D_{p_{k_n}, q_{j}^{(k_n)}}(\zeta)$ for every $j = 1, \cdots, M(k_n)$.

\end{proof}

Next we consider a finite of infinite denumerable family of systems:

$$ \mathcal{A}^{(l)} = ((p^{(l)}_n)_{n \geq 1}, N(l, n), q^{(n)}_{l, i} \; \text{for} \; i = 1, \cdots, N(l, n)), l \in I $$

where $I = \mathbb{N}$ or $I$ is finite. As expected, for every $l \in I$ it holds $(p^{(l)}_n)_{n \geq 1} \subseteq \mathbb{N}$, $p^{(l)}_n \to + \infty$ as $n \to + \infty$, $N(l, n) \in \mathbb{N}$ for every $n \in \mathbb{N}$, $q^{(n)}_{l, i} \in \mathbb{N}$ for every $i = 1, \cdots, N(l, n)$ and $\max \{ q^{(n)}_{l, 1}, \cdots, q^{(n)}_{l, N(l, n)} \} \to + \infty$ as $n \to + \infty$.

Each system defines a new class of functions, namely the class $\mathcal{B}(\mathcal{A}^{(l)})$. This is done in a similar way to the class $\mathcal{B}(\mathcal{A}^{(1)})$ we constructed in Theorem \ref{theorem 5.1}.

We will now show that the class $\cap_{l \in I} \mathcal{B}(\mathcal{A}^{(l)})$ is a dense subset of $H(\Omega)$, for $I$ a finite or an infinite denumerable set.

\begin{theorem} \label{theorem 5.5} Let $I \neq \emptyset$ be a finite or infinite denumerable set. We consider a family of systems $\{ \mathcal{A}^{(l)} \}_{l \in I}$ as above. Then the class $\cap_{l \in I} \mathcal{B}(\mathcal{A}^{(l)})$ is a dense subset of $H(\Omega)$ and every function $f \in \cap_{l \in I} \mathcal{B}(\mathcal{A}^{(l)})$, $f(z) = \sum_{n = 0}^{+ \infty} a_n (z - \zeta)^n$ satisfies the following:

For every compact set $K \subseteq \mathbb{C} \setminus \Omega$ with connected complement, for every function $h \in A(K)$ and for every $l \in I$, there exists a subsequence $(p^{(l)}_{k_n(l)})_{n \geq 1}$ of the sequence $(p^{(l)}_n)_{n \geq 1}$ such that:

\begin{itemize}

\item[(1)]
$\sup_{z \in K} |S_{p^{(l)}_{k_n(l)}}(f, \zeta)(z) - h(z)| \to 0$, as $n \to + \infty$.

\item[(2)]
$\sup_{z \in J} |S_{p^{(l)}_{k_n(l)}}(f, \zeta)(z) - f(z)| \to 0$, as $n \to + \infty$ for every compact set $J \subseteq \Omega$.

\end{itemize}

Furthermore, for every $l \in I$ and for every $n \in \mathbb{N}$ it holds $a_{p^{(l)}_{k_n(l)}} \neq 0$  and $a_{p^{(l)}_{k_n(l)} + s} = 0$ for every $s = 1, \cdots, \max\{ q^{(k_n(l))}_{l, 1}, \cdots, q^{(k_n(l))}_{l, N(l, k_n(l))}\}$.

\end{theorem}

\begin{proof} The proof is based on the one of Theorem \ref{theorem 5.1}. We examine each one of the two cases seperately.

\begin{itemize}

\item[(1)] The set $I$ is finite. In this case, without loss of the generality, we suppose that it holds $I = \{ 1, 2, \cdots, N \}$, for an index $N \in \mathbb{N}$. 

In each step of the proof we repeat the same arguments as we did in the one of Theorem \ref{theorem 5.1}, but this time the specific argument is used for a different sequence. Refer to the following table for details.

\bigskip

\begin{tabular}{ c | c c c c}
& Selection & Selection & $\cdots$ & Selection \\
Sequence & $(p^{(1)}_n)_{n \geq 1}$ & $(p^{(2)}_n)_{n \geq 1}$ &  $\cdots$ & $(p^{(N)}_n)_{n \geq 1}$ \\ \hline    
Step $1$ & $\bullet$ & - & $\cdots$ & - \\
Step $2$ & - & $\bullet$ & $\cdots$ & - \\
$\vdots$  & $\vdots$  & $\vdots$ & $\ddots$ & $\vdots$ \\
Step $N$ & - & - & $\cdots$ & $\bullet$ \\
Step $N + 1$ & $\bullet$ & - & $\cdots$ & - \\
Step $N + 2$ & - & $\bullet$ & $\cdots$ & - \\
$\vdots$  & $\vdots$  & $\vdots$ & $\ddots$ & $\vdots$ \\
Step $2N$ & - & - & $\cdots$ & $\bullet$ \\
$\vdots$  & $\vdots$  & $\vdots$ & $\cdots$ & $\vdots$ \\
\end{tabular}

\bigskip

The class $\cap_{l \in I} \mathcal{B}(\mathcal{A}^{(l)})$ is proved to be a dense subset of $H(\Omega)$ in the same way. In addition, every function $f \in \cap_{l \in I} \mathcal{B}(\mathcal{A}^{(l)})$ meets the requirements of Theorem \ref{theorem 5.5} in the same way as in Theorem \ref{theorem 5.1}.

\item[(2)] The set $I$ is infinite denumerable. In this case, without loss of the generality, we suppose that it holds $I = \mathbb{N}$. 

In each step of the proof we repeat the same arguments as we did in the one of Theorem \ref{theorem 5.1}, but this time the specific argument is used for a different sequence. Refer to the following table for details.

\bigskip

\begin{tabular}{ c | c c c c c c}
& Selection & Selection & Selection & $\cdots$ & Selection & $\cdots$ \\
Sequence & $(p^{(1)}_n)_{n \geq 1}$ & $(p^{(2)}_n)_{n \geq 1}$ & $(p^{(3)}_n)_{n \geq 1}$ & $\cdots$ & $(p^{(N)}_n)_{n \geq 1}$ & $\cdots$ \\ \hline    
Step $1$ & $\bullet$ & - & - & $\cdots$ & - & $\cdots$ \\
Step $2$ & $\bullet$ & $\bullet$ & - & $\cdots$ & - & $\cdots$ \\
$\vdots$  & $\vdots$ & $\vdots$ & $\vdots$ & $\ddots$ & $\vdots$ & $\cdots$ \\
Step $N$ & $\bullet$ & $\bullet$ & $\cdots$ & $\cdots$ & $\bullet$ & $\cdots$ \\
Step $N + 1$ & $\bullet$ & $\bullet$ & $\bullet$ & $\cdots$ & $\bullet$ & $\bullet$ \\
$\vdots$  & $\vdots$  & $\vdots$ & $\cdots$ & $\vdots$ & $\vdots$ & $\cdots$ \\
\end{tabular}

\bigskip

The class $\cap_{l \in I} \mathcal{B}(\mathcal{A}^{(l)})$ is proved to be a dense subset of $H(\Omega)$ in the same way. In addition, every function $f \in \cap_{l \in I} \mathcal{B}(\mathcal{A}^{(l)})$ meets the requirements of Theorem \ref{theorem 5.5} in the same way as in Theorem \ref{theorem 5.1}.

\end{itemize}

\end{proof}

\begin{proposition} \label{proposition 5.6}

Let $I \neq \emptyset$ be a finite or infinite denumerable set, $\{ \mathcal{A}^{(l)} \}_{l \in I}$ systems as above, $g \in \cap_{l \in I} \mathcal{B}(\mathcal{A}^{(l)})$ and $P$ be a polynomial. Then the function $g + P \in \cap_{l \in I} \mathcal{B}(\mathcal{A}^{(l)})$. It follows that the class $\cap_{l \in I} \mathcal{B}(\mathcal{A}^{(l)})$ is a dense subset of $H(\Omega)$.

\end{proposition}

\begin{proof} The proof is obvious according to the definition of the classes $\mathcal{B}(\mathcal{A}^{(l)})$, $l \in I$.

\end{proof}

We give now the definition of a new class of functions, namely the class $\mathcal{A}'$, which is larger than the class $\mathcal{A}$; i.e. $\mathcal{A} \subseteq \mathcal{A}'$.

\begin{definition} \label{definition 5.7}

Let $(p_n)_{n \in \mathbb{N}} \subseteq \mathbb{N}$ with $p_n \to + \infty$. For every $n \in \mathbb{N}$, let $q^{(n)}_1, \cdots, q^{(n)}_{N(n)} \in \mathbb{N}$, where $N(n) \in \mathbb{N}$. Let $\Omega \subseteq \mathbb{C}$ be a simply connected domain and $\zeta \in \Omega$ be a fixed element. Then there exists a holomorphic function $f \in H(\Omega)$ with Taylor expansion $f(z) = \sum_{n = 0}^{+ \infty} \frac{f^{(n)}(\zeta)}{n!} (z - \zeta)^{n}$ such that for every polynomial $h$ and for every compact set $K \subseteq \mathbb{C} \setminus \Omega$ with connected complement there exists a subsequence $(p_{k_n})_{n \in \mathbb{N}}$ of the sequence $(p_n)_{n \in \mathbb{N}}$ so that the following hold:

\begin{itemize}

\item[(i)]
The Pad\'{e} approximant $[f; p_{k_n} / q^{(k_n)}_{\sigma(k_n)}]_{\zeta}$ exists for every $n \in \mathbb{N}$ and for every selection $\sigma: \mathbb{N} \to \mathbb{N}$ satisfying $\sigma(k_n) \in \{ 1, \cdots, N(k_n) \}$.

\item[(ii)]
$\sup_{z \in K} |[f; p_{k_n} / q^{(k_n)}_{\sigma(k_n)}]_{\zeta}(z) - h(z)| \to 0$ for every selection $\sigma: \mathbb{N} \to \mathbb{N}$ satisfying $\sigma(k_n) \in \{ 1, \cdots, N(k_n) \}$.

\item[(iii)]
$\sup_{z \in J} |[f; p_{k_n} / q^{(k_n)}_{\sigma(k_n)}]_{\zeta}(z) - f(z)|$ for every compact set $J \subseteq \Omega$ and for every selection $\sigma: \mathbb{N} \to \mathbb{N}$ satisfying $\sigma(k_n) \in \{ 1, \cdots, N(k_n) \}$.

\end{itemize}

The class $\mathcal{A}'$ is defined as the set of all functions satisfying Definition \ref{definition 5.7}. In addition it holds $\mathcal{A} \subseteq \mathcal{A}'$, according to the definition of the class $\mathcal{A}$.

\end{definition}

The main result of this section is that there exists a dense subspace $V$ of $H(\Omega)$ such that $V \setminus \{ 0 \} \subseteq \mathcal{A}'$. This implies that the class $\mathcal{A}'$ is algebraically generic in $H(\Omega)$ and strengthens a result of \cite{CHARPENTIER.NESTORIDIS.WIELONSKY}.

\begin{theorem} \label{theorem 5.8}

There exists a dense vector subspace $V$ of $H(\Omega)$ such that $V \setminus \{ 0 \} \subseteq \mathcal{A}'$.

\end{theorem}

\begin{proof} Let $\{ K_m \}_{m \geq 1}$ and $\{ L_k \}_{k \geq 1}$ be two families of compact subsets of $\mathbb{C}$ satisfying Lemmas \ref{lemma 3.1} and \ref{lemma 3.2} respectively. Let also $\{ f_i \}_{i \geq 1}$ be an enumeration of polynomials with coefficients in $\mathbb{Q} + i \mathbb{Q}$.

\begin{itemize}

\item[\textbf{Step 1.}] We consider a function $g_1 \in \mathcal{B}(\mathcal{A}^{(1)})$ so that the following hold:

\begin{itemize}

\item[(1)] $\rho(g_1, f_1) < \frac{1}{1}$.

\item[(2)] For every $m \geq 1$ there exists a subsequence $(p^{(1)}_{m, n})_{n \geq 1}$ of the sequence $(p_n)_{n \geq 1}$ such that:
$$ \sup_{z \in K_m} |S_{p^{(1)}_{m, n}}(g_1, \zeta)(z) - 0| \to 0 \; \text{as} \; n \to + \infty. $$

$$ \sup_{z \in J} |S_{p^{(1)}_{m, n}}(g_1, \zeta)(z) - g_1(z)| \to 0 \; \text{as} \; n \to + \infty $$

for every compact $J \subseteq \Omega$.

\end{itemize}

So at the end of \textbf{Step 1} we have constructed infinitely many subsequences of the sequence $(p_n)_{n \geq 1}$; the sequences $(p^{(1)}_{m, n})_{n \geq 1}$ for $m \geq 1$.

\item[\textbf{Step 2.}] We consider the system:
$$ \mathcal{A}^{((1), m)} = ((p^{(1)}_{m, n})_{n \geq 1}, N(((1), m), n), q^{(n)}_{i, ((1), m)} $$
$$ \text{for} \; i = 1, \cdots, N(((1), m), n), m \in \mathbb{N}. $$

According to Theorem \ref{theorem 5.5}, we yield the following:

There exists a function $g_2 \in \cap_{m \in \mathbb{N}} \mathcal{B}(\mathcal{A}^{((1), m)})$ so that the following hold:

\begin{itemize}

\item[(1)] $\rho(g_2, f_2) < \frac{1}{2}$.

\item[(2)] For every $m \geq 1$ there exists a subsequence $(p^{(2)}_{m, n})_{n \geq 1}$ of the sequence $(p^{(1)}_{m, n})_{n \geq 1}$ such that:
$$ \sup_{z \in K_m} |S_{p^{(2)}_{m, n}}(g_2, \zeta)(z) - 0| \to 0 \; \text{as} \; n \to + \infty. $$

$$ \sup_{z \in J} |S_{p^{(2)}_{m, n}}(g_2, \zeta)(z) - g_2(z)| \to 0 \; \text{as} \; n \to + \infty $$

for every compact $J \subseteq \Omega$.

\end{itemize}

So at the end of \textbf{Step 2} we have constructed infinitely many subsequences of the sequence $(p_n)_{n \geq 1}$, since for every $m \geq 1$ the sequence $(p^{(2)}_{m, n})_{n \geq 1}$ is a subsequence of $(p^{(1)}_{m, n})_{n \geq 1}$.

\item[\textbf{Step N.}] We consider the system:
$$ \mathcal{A}^{((N), m)} = ((p^{(N - 1)}_{m, n})_{n \geq 1}, N(((N - 1), m), n), q^{(n)}_{i, ((N - 1), m)} $$
$$ \text{for} \; i = 1, \cdots, N(((N - 1), m), n), m \in \mathbb{N}. $$ 

According to Theorem \ref{theorem 5.5}, we yield the following:

There exists a function $g_N \in \cap_{m \in \mathbb{N}} \mathcal{B}(\mathcal{A}^{((N), m)})$ so that the following hold:

\begin{itemize}

\item[(1)] $\rho(g_N, f_N) < \frac{1}{N}$.

\item[(2)] For every $m \geq 1$ there exists a subsequence $(p^{(N)}_{m, n})_{n \geq 1}$ of the sequence $(p^{(N - 1)}_{m, n})_{n \geq 1}$ such that:
$$ \sup_{z \in K_m} |S_{p^{(N)}_{m, n}}(g_N, \zeta)(z) - 0| \to 0 \; \text{as} \; n \to + \infty. $$

$$ \sup_{z \in J} |S_{p^{(N)}_{m, n}}(g_N, \zeta)(z) - g_N(z)| \to 0 \; \text{as} \; n \to + \infty $$

for every compact $J \subseteq \Omega$.

\end{itemize}

So at the end of \textbf{Step N} we have constructed infinitely many subsequences of the sequence $(p_n)_{n \geq 1}$, since for every $m \geq 1$ the sequence $(p^{(N)}_{m, n})_{n \geq 1}$ is a subsequence of $(p^{(N - 1)}_{m, n})_{n \geq 1}$.

\end{itemize}

We consider now the linear span $<g_n : n \geq 1> \; \subseteq H(\Omega)$. Let $j_1 < \cdots < j_s \in \mathbb{N}$ and $a_{j_1}, \cdots, a_{j_s} \in \mathbb{C} \setminus \{ 0 \}$. We set $g \equiv a_{j_1} g_{j_1} + \cdots + a_{j_s} g_{j_s}$. Our aim is to prove that the function $g$ is a universal Taylor series and belong to the class $\mathcal{A'}$. Thus, let $K \subseteq \mathbb{C} \setminus \Omega$ be a compact set with connected complement and $h \in A(K)$. We consider an index $m \in \mathbb{N}$ such that $K \subseteq K_m$. Since $g_{j_s} \in \cap_{m \in \mathbb{N}} \mathcal{B}(\mathcal{A}^{((j_s), m)})$, there exists a subsequence $(p^{(j_s)}_{k_n})_{n \geq 1}$ of the sequence $(g^{(j_s)}_{m, n})_{n \geq 1}$ such that:

\begin{itemize}

\item[(1)] $\sup_{z \in K_m} |S_{p^{(j_s)}_{k_n}}(g_{j_s}, \zeta)(z) - \frac{h(z)}{a_{j_s}}| \to 0 \; \text{as} \; n \to + \infty$.

\item[(2)] $\sup_{z \in J} |S_{p^{(j_s)}_{k_n}}(g_{j_s}, \zeta)(z) - g_{j_s}(z)| \to 0 \; \text{as} \; n \to + \infty$

for every compact set $J \subseteq \Omega$.

\end{itemize}

Since $(p^{(j_s)}_{k_n})_{n \geq 1}$ is a subsequence of the sequence $(g^{(j_s)}_{m, n})_{n \geq 1}$ we obtain that for every $t < s$ it holds:

\begin{itemize}

\item[(1)] $\sup_{z \in K_m} |S_{p^{(j_s)}_{k_n}}(g_{j_t}, \zeta)(z) - 0| \to 0 \; \text{as} \; n \to + \infty$.

\item[(2)] $\sup_{z \in J} |S_{p^{(j_s)}_{k_n}}(g_{j_t}, \zeta)(z) - g_{j_t}(z)| \to 0 \; \text{as} \; n \to + \infty$

for every compact set $J \subseteq \Omega$.

\end{itemize}

Thus, it follows that:

\begin{itemize}

\item[(1)] $$ \sup_{z \in K_m} |S_{p^{(j_s)}_{k_n}}(g, \zeta)(z) - h(z)| = $$

$$ = \sup_{z \in K_m} |S_{p^{(j_s)}_{k_n}}(a_{j_1} g_{j_1} + \cdots + a_{j_s} g_{j_s}, \zeta)(z) - h(z)| = $$

$$ = \sup_{z \in K_m} |a_{j_1} S_{p^{(j_s)}_{k_n}}(g_{j_1}, \zeta)(z) + \cdots + a_{j_s} S_{p^{(j_s)}_{k_n}}(g_{j_s}, \zeta)(z) - h(z)| \leq $$

$$ \leq \sup_{z \in K_m} |a_{j_1} S_{p^{(j_s)}_{k_n}}(g_{j_1}, \zeta)(z)| + \cdots + \sup_{z \in K_m} |a_{j_{s - 1}} S_{p^{(j_s)}_{k_n}}(g_{j_{s - 1}}, \zeta)(z)| + $$

$$ + \sup_{z \in K_m} |a_{j_1} S_{p^{(j_s)}_{k_n}}(g_{j_s}, \zeta)(z) - h(z)| \to 0 \; \text{as} \; n \to + \infty. $$

\item[(2)] $\sup_{z \in J} |S_{p^{(j_s)}_{k_n}}(g, \zeta)(z) - g(z)| \to 0 \; \text{as} \; n \to + \infty$

for every compact set $J \subseteq \Omega$.

\end{itemize}

Finally, we notice that for every $n \geq 1$ it holds: $[g; p^{(j_s)}_{k_n} / q^{(k_n)}_{(j_s), i}]_{\zeta}(z) = S_{p^{(j_s)}_{k_n}}(g, \zeta)(z)$ for every $i = 1, \cdots N((j_s), k_n)$.

\end{proof}

\begin{remark} \label{remark 5.9} In the above construction we may have that $a_{p_{k_n}} = 0$ for some $n \in \mathbb{N}$, which would imply that $g \not\in D_{p_{k_n}, p_{k_n}}(\zeta)$ and also that $g \not\in \mathcal{A}$. But still we have that $g \in \mathcal{A'}$.

\end{remark}

\begin{remark} \label{remark 5.10} If $q^{(i)}_n \neq 0$ then every function belonging to the class $\mathcal{A'}$ has some Taylor coefficients equal to zero. It follows that $\mathcal{A'}$ is meager in $H(\Omega)$. Since the set of universal Taylor series is dense and $G_{\delta}$ in $H(\Omega)$ it follows that the result of Theorem \ref{theorem 5.8} can not be deduced from the known results about algebraic genericity of the set of universal Taylor series.

\end{remark}

\bigskip
\bigskip
\bigskip

\noindent
University of Athens
\\
Department of Mathematics
\\
157 84 Panepistimioupolis
\\
Athens
\\
Greece

\bigskip

\noindent
Email Address:
\\
K. Makridis (kmak167@gmail.com)

\noindent
V. Nestoridis (vnestor@math.uoa.gr)

\end{document}